\documentclass[12pt]{article}
\usepackage{amssymb,amsthm,amsmath,amsfonts,latexsym,tikz,hyperref}
\usepackage[hmargin=1in,vmargin=1in]{geometry}
\usepackage{shuffle}
\tikzset{node distance=2cm, auto} 
\usetikzlibrary{snakes}

\newtheorem{thm}{Theorem}[section]
\newtheorem{prop}[thm]{Proposition}
\newtheorem{cor}[thm]{Corollary}
\newtheorem{lem}[thm]{Lemma}
\newtheorem{conj}[thm]{Conjecture}
\newtheorem{exa}[thm]{Example}
\newtheorem{quest}[thm]{Question}
\theoremstyle{definition}
\newtheorem{defn}[thm]{Definition}

\DeclareMathOperator{\std}{std} 
\DeclareMathOperator{\dom}{dom} 
\DeclareMathOperator{\QSym}{QSym}

\DeclareMathOperator{\St}{St} 
\DeclareMathOperator{\inver}{inv} 
\DeclareMathOperator{\Pk}{Pk} 
\DeclareMathOperator{\Rpk}{Rpk} 
\DeclareMathOperator{\Lpk}{Lpk} 
\DeclareMathOperator{\rpk}{rpk} 
\DeclareMathOperator{\lpk}{lpk} 
\DeclareMathOperator{\pk}{pk} 
\DeclareMathOperator{\Epk}{Epk} 
\DeclareMathOperator{\Eval}{Eval} 
\DeclareMathOperator{\eval}{eval} 
\DeclareMathOperator{\epk}{epk} 
\DeclareMathOperator{\Val}{Val} 
\DeclareMathOperator{\Lval}{Lval} 
\DeclareMathOperator{\Rval}{Rval} 
\DeclareMathOperator{\lval}{lval} 
\DeclareMathOperator{\rval}{rval} 
\DeclareMathOperator{\val}{val} 
\DeclareMathOperator{\maj}{maj} 
\DeclareMathOperator{\Des}{Des} 
\DeclareMathOperator{\Asc}{Asc} 
\DeclareMathOperator{\des}{des}
\DeclareMathOperator{\asc}{asc}
\DeclareMathOperator{\udr}{udr}

\newcommand{\ben}{\begin{enumerate}}
\newcommand{\een}{\end{enumerate}}
\newcommand{\ble}{\begin{lem}}
\newcommand{\ele}{\end{lem}}
\newcommand{\bth}{\begin{thm}}
\renewcommand{\eth}{\end{thm}}
\newcommand{\bpr}{\begin{prop}}
\newcommand{\epr}{\end{prop}}
\newcommand{\bco}{\begin{cor}}
\newcommand{\eco}{\end{cor}}
\newcommand{\bcon}{\begin{conj}}
\newcommand{\econ}{\end{conj}}
\newcommand{\bde}{\begin{defn}}
\newcommand{\ede}{\end{defn}}
\newcommand{\bex}{\begin{exa}}
\newcommand{\eex}{\end{exa}}
\newcommand{\barr}{\begin{array}}
\newcommand{\earr}{\end{array}}
\newcommand{\btab}{\begin{tabular}}
\newcommand{\etab}{\end{tabular}}
\newcommand{\beq}{\begin{equation}}
\newcommand{\eeq}{\end{equation}}
\newcommand{\bea}{\begin{eqnarray*}}
\newcommand{\eea}{\end{eqnarray*}}
\newcommand{\bal}{\begin{align*}}
\newcommand{\bce}{\begin{center}}
\newcommand{\ece}{\end{center}}
\newcommand{\bpi}{\begin{picture}}
\newcommand{\epi}{\end{picture}}
\newcommand{\bpp}{\begin{picture}}
\newcommand{\epp}{\end{picture}}
\newcommand{\bfi}{\begin{figure} \begin{center}}
\newcommand{\efi}{\end{center} \end{figure}}
\newcommand{\bprf}{\begin{proof}}
\newcommand{\eprf}{\end{proof}\medskip}

\newcommand{\bsl}{\begin{slide}{}}
\newcommand{\esl}{\end{slide}}
\newcommand{\bfr}{\begin{frame}}
\newcommand{\efr}{\end{frame}}

\newcommand{\hso}[1]{\hspace{-1pt}}

\newcommand{\setm}{\setminus}

\def\<{\langle}
\def\>{\rangle}

\newcommand{\ree}[1]{(\ref{#1})}

\newcommand{\si}{\sigma}

\newcommand{\bbN}{{\mathbb N}}
\newcommand{\bbP}{{\mathbb P}}

\begin{document}
\pagestyle{plain}

\title{Bijective proofs of shuffle compatibility results
}
\author{Duff Baker-Jarvis\\[-5pt]
\small Department of Mathematics, Michigan State University,\\[-5pt]
\small East Lansing, MI 48824-1027, USA, {\tt bakerjar@msu.edu}\\
Bruce E. Sagan\\[-5pt]
\small Department of Mathematics, Michigan State University,\\[-5pt]
\small East Lansing, MI 48824-1027, USA, {\tt sagan@math.msu.edu}
}

\date{\today\\[10pt]
	\begin{flushleft}
	                                       \\[5pt]
	\end{flushleft}}

\maketitle

\begin{abstract}

Define a permutation to be any sequence of distinct positive integers.  Given two permutations $\pi$ and $\sigma$ on disjoint underlying sets, we denote by $\pi\shuffle\sigma$ the set of shuffles of $\pi$ and $\sigma$ (the set of all permutations obtained by interleaving the two permutations).  A permutation statistic is a function $\St$ whose domain is the set of permutations such that $\St(\pi)$ only depends on the relative order of the elements of $\pi$.  A permutation statistic is shuffle compatible if the distribution of $\St$ on $\pi\shuffle\sigma$ depends only on $\St(\pi)$ and $\St(\sigma)$ and their lengths rather than on the individual permutations themselves.  This notion is implicit in the work of Stanley in his theory of $P$-partitions.
The definition was explicitly given by Gessel and Zhuang who proved that various permutation statistics were shuffle compatible using mainly algebraic means.  This work was continued by Grinberg. 
The purpose of the present article is to use bijective techniques to give demonstrations of shuffle compatibility.  In particular, we show how a large number of permutation statistics can be shown to be shuffle compatible using a few simple bijections.  Our approach also leads to a method for constructing such bijective proofs rather than having to treat each one in an ad hoc manner.  Finally, we are able to prove a conjecture of Gessel and Zhuang about the shuffle compatibility of a certain statistic.

\end{abstract}

%
%
%

\section{Introduction}
Let $\bbP$ and $\bbN$ be the positive and nonnegative integers, respectively. To denote the cardinality of a set $U$ we use $\#U$ or $|U|$. All subsets of $\bbP$ should be assumed to be finite unless otherwise noted. For  $U\subset \bbP$, a \emph{permutation} of $U$ is a linear order $\pi=\pi_1\pi_2\ldots\pi_n$ of the elements of $U$.  
We sometimes separate the elements of $\pi$ by commas for ease of reading.
Denote the set of all linear orders on $U$ by 
$$L(U) = \{ \pi \; | \; \pi \text{ is a linear order on } U\}.$$ The \emph{length} of a permutation is the cardinality of its underlying set, i.e. $|U|$, which we denote by $|\pi|$. The \emph{domain} of a permutation $\pi \in L(U)$ is the set $U$, and we write 
$\dom(\pi)=U$. 

For $n,i,j\in \bbP$  we use the notation  $[n] = \{1, 2, \ldots, n\}$, and $[i,j]=\{i, i+1, \ldots, j\}$. Also let $[n]+i = \{k+i \; | \; k\in [n]\}$.  For sets, $U\sqcup V=W$ indicates that $W$ is the disjoint union of $U$ and $V$. 
Double braces indicate a multiset, that is, a family of elements where repetition is allowed. We will sometimes use multiplicity notation for multisets, e.g., $\{\{1,2^4,3^3\}\}$ is the multiset that contains $1$, $2$ four times, and $3$ three times.

 To compare permutations on different sets of the same size we have the following definition.
 
\begin{defn}
Let $U,V\subseteq \bbP$ be two subsets of the positive integers such that $|U|=|V|$. Let $\pi\in L(U)$. Define the \emph{standardization to $V$} of $\pi=\pi_1\pi_2\ldots \pi_{\ell}$ to be
$$\std_{V}(\pi) = f(\pi_1)f(\pi_2)\ldots f(\pi_\ell)$$
where $f:U\to V$ is the unique strictly increasing bijection. Let $n=|U|$. Then, if no subscript is given, define
$$\std\pi:=\std_{[n]}(\pi)$$
to be the standardization of $\pi$ to $\{1,2, \ldots, |U|\}$. \hfill $\square$
\end{defn}

Two permutations $\pi\in L(U)$, and $ \pi'\in L(V)$ are said to have the same \emph{relative order} if $\std_{V}(\pi)=\pi'$, or equivalently, $\std_U(\pi')=\pi$. For example if $U = \{1,7,8\}$, $V = \{2,3,9\}$, $\pi = 781$, $\pi'=392$, then $\pi$ and $\pi'$ have the same relative order since $\std_{U}(392) =781$. Equivalently, $\std\pi= \std\pi' = 231$. 

A \emph{permutation statistic} is a map $\St$ with domain 
$$\bigsqcup \limits_{\substack{U\subset \bbP\\ |U|<\infty}} L[U]$$
such that whenever $\pi$ and $\pi'$ have the same relative order, then $\St(\pi)=\St(\pi')$. It is useful to extend the notation for permutation statistics to sets of permutations by defining, for a set of permutations $\Pi$, $\St(\Pi)$ to be the multiset
$$\St(\Pi) = \{\{\St(\pi) \mid \pi\in \Pi \}\}.$$
We call this multiset the \emph{distribution} of $\St$ over the set $\Pi$.

The basic example of a (set-valued) permutation statistic is that of the descent set, $\Des$. For a permutation $\pi$, a \emph{descent} of $\pi$ is a position $i$ such that $\pi_i>\pi_{i+1}$. Then the \emph{descent set of $\pi$} is
$$\Des(\pi) = \{i  \mid i \text{ is a descent of } \pi\}.$$
The \emph{descent number} of $\pi$ is $\des(\pi) = \#\Des(\pi)$. Another important statistic is the \emph{major index}, $\maj$, given by
$$\maj(\pi) = \sum \limits_{i\in \Des(\pi)}i.$$
For example, given the permutation $\pi = 2157364\in L([7])$ we have $\Des(\pi) = \{1,4,6\}$ and $\maj(\pi) = 1+4+6=11$.

 We call a permutation statistic, $\St$,  a \emph{descent statistic} if it is a permutation statistic such that $\Des(\pi)=\Des(\pi')$ implies $\St(\pi)=\St(\pi')$. Both $\Des$ and $\maj$ are examples of descent statistics. There are many permutation statistics in the literature which are not descent statistics. One such statistic is 
 $$\inver(\pi) = \#\{i<j \mid \pi_i>\pi_j\}$$
 which counts the number of inversions in a permutation. For example $132$ and $231$ are two permutations that have the same descent set $\{2\}$, but $\inver(132) = 1$ whereas $\inver(231) = 2$.

Given a permutation $\pi$, a {\em subword} of $\pi$ is a subsequence of not necessarily consecutive elements, whereas a {\em factor} is a subsequence whose elements are consecutive.
For two permutations with disjoint domains, a \emph{shuffle} of $\pi$ and $\sigma$ is  a permutation $\tau\in L(\dom(\pi)\sqcup\dom(\sigma))$ 
such that both $\pi$ and $\sigma$ occur as subwords. The \emph{shuffle set} of $\pi$ and $\sigma$ is
 	$$\pi \shuffle \sigma = \{ \tau \; | \; \tau \text{ is a shuffle of $\pi$, $\sigma$}\}$$
which always has cardinality $\binom{|\pi|+|\sigma|}{|\pi|}$. As an example, if $\pi=132$ and $\sigma=76$ then
$$\pi \shuffle \sigma = \{13276,\ 13726,\ 13762,\ 17326,\ 17362,\ 17632,\ 71326,\ 71362,\ 71632, 76132 \}$$
which has size $\binom{3+2}{3} = 10$. Whenever we write a shuffle set $\pi \shuffle \sigma$ we will implicitly assume that the permutations $\pi$ and $\sigma$ have disjoint domains.  It is an interesting fact that the statistics $\maj$ and $\inver$ have the same distribution over the shuffle set $\pi\shuffle \sigma$.

It is helpful to have a way to discuss and distinguish individual shuffles without carrying the entire information of both permutations along. To do this, we will use words in the two letter alphabet $A=\{a,b\}$. Denote  by
$$A^* = \{\alpha_1\alpha_2\ldots \alpha_k \ | \ k\geq 0, \ \alpha_i\in A\}$$
the set of words in the letters of $A$. This is  called the \emph{Kleene closure} of $A$. Suppose that $|\pi|=m$ and $|\sigma|=n$. 
Then we will associate to each permutation $\tau\in \pi\shuffle \sigma$ a word $\omega(\tau)\in A^*$ of length $m+n$ obtained by replacing the elements of $\pi$ with the letter $a$ and elements of $\sigma$ with the letter $b$.
Call $\omega(\tau)$ the \emph{word of $\tau$}.
 For example if $\pi=132$, $\sigma=4589$ and $\tau=1453829$, then $\omega(\tau) =  abbabab$. It is true that $\omega(\tau)$ depends on both $\pi$ and $\sigma$, but these will always be clear from context.
We now introduce the definition which will be our fundamental object of study.

\begin{defn}
\label{defn-shufcompat}
Assume $\pi,\pi', \sigma, \sigma'$ are permutations such that $|\pi|=|\pi'|$, $|\sigma|=|\sigma'|$ and $\dom(\pi)\cap \dom(\sigma)=\dom(\pi')\cap \dom(\sigma')=\emptyset$. Call a permutation statistic $\St$ \emph{shuffle compatible} if  for all such permutations which also satisfy $\St(\pi)=\St(\pi')$ and $\St(\sigma)=\St(\sigma')$ we have
$$\St(\pi\shuffle \sigma) = \St(\pi'\shuffle \sigma').$$
Being shuffle compatible is also equivalent to the existence of a bijection $\Theta:\pi\shuffle \sigma \to \pi' \shuffle \sigma'$ with the property that $\St(\Theta(\tau))=\St(\tau)$ for all $\tau\in\pi\shuffle \sigma$. We call a map with this property \emph{$\St$ preserving}. 
\hfill $\square$
\end{defn}

As an illustration of this, let us consider the $\maj$ statistic with $\pi = 4312$, $\pi' = 2341$, $\sigma=76$ and $\sigma'= 98$. These satisfy $\maj(\pi)=\maj(\pi') = 3$ and $\maj(\sigma)=\maj(\sigma')=1$.
Then one can check that 
$$
\maj(\pi\shuffle \sigma) = \maj(\pi'\shuffle \sigma') = \{\{4, 5, 6^2, 7^2, 8^3, 9^2, 10^2, 11, 12\}\}.
$$
The $\maj$ statistic is indeed shuffle compatible as we will show later.

The standard {\em $q$-analogue} of the nonnegative integer $n$ is 
$$
[n]_q = 1+q+q^2+\ldots +q^{n-1}
$$ 
where $q$ is a variable. Given integers $0\le k \le n$ the corresponding {\em $q$-binomial coefficient} is given by
$${n \brack k}_q = \frac{[n]_q!}{[k]_q![n-k]_q!}$$
where
$$
[n]_q! =  [1]_q [2]_q\cdots[n]_q.
$$ 
Let 
$$\pi \shuffle_k \sigma = \{\tau\in \pi \shuffle \sigma \; | \; \des(\tau) =k \}.$$
Stanley~\cite{Stanley1}, gave proofs of the identities
$$\sum \limits_{\tau\in \pi\shuffle \sigma} q^{\maj\tau} = q^{\maj\pi+\maj\sigma}{|\pi|+|\sigma| \brack |\pi|}_q$$
and
$$\sum \limits_{\tau\in \pi \shuffle_k \sigma}q^{\maj \tau} = q^{\maj\pi +\maj\sigma +(k-\des\pi)(k-\des\sigma)} {|\pi|-\des(\pi)+\des(\sigma) \brack k-\des(\pi)}_q {|\sigma|-\des(\sigma)+\des(\pi) \brack k-\des(\sigma)}_q$$ 
where $\pi \shuffle_k \sigma$ is the set of all shuffles of $\pi$ and $\sigma$ which have $k$ descents.
These imply that  $\maj$ and $\des$ are shuffle compatible. He utilized $P$-partitions to obtain them, and later bijective proofs were given by Goulden~\cite{Goulden85}, Guha and Padmanabhan~\cite{GuhaPad1990}, and Novick~\cite{Novick2010}.

  A recent paper by Gessel and Zhuang \cite{GesZhu17} introduced the idea of a shuffle compatible permutation statistic and proceeded to show that many permutation statistics in fact do have this property. In addition they showed that a descent statistic being shuffle compatible is equivalent to the existence of a certain algebra 
that is a quotient algebra of the Hopf algebra $\QSym$ of quasisymmetric functions. 
The algebra $\QSym$ can itself be identified as the shuffle algebra of the descent set statistic $\Des$.
The methods of Gessel and Zhuang were primarily algebraic using noncommutative symmetric functions, quasisymmetric functions, and variants of quasisymmetric functions to prove that statistics were shuffle compatible. They were also able to characterize many of the  algebras corresponding to these statistics. They also conjectured that several permutation statistics were shuffle compatible. Some of these conjectures were then proven by Grinberg in \cite{Grin2018} using enriched $P$-partitions similar to those developed by Stembridge in \cite{Stembridge1997}.
It was also conjectured in \cite{GesZhu17} that perhaps it was true that all shuffle compatible permutation statistic descent statistics. This has been shown to be false as an example of a shuffle compatible permutation statistic that is not a descent statistic was given by O{\u{g}}uz in \cite{Oguz2018}.

In this paper we present  a bijective approach to showing that permutation statistics are shuffle compatible. Our method has the  following three advantages.  First of all, it is uniform in that essentially the same steps are followed to achieve each result. In addition, our proofs tend to be shorter and more transparent than other methods.  Finally, we are also able to prove shuffle compatibility for $(\udr, \pk)$, one of the statistics conjectured to be shuffle compatible by Gessel and Zhuang which have resisted other techniques.

The rest of this paper is structured as follows. Section~2 gives a summary of the definitions of the various permutation statistics that we study. In Section~3 we outline our general approach  to proving shuffle compatibility as well as proving one of the main reductions that is  repeatedly used. Bijective proofs for the shuffle compatibility of the known shuffle compatible set valued statistics are given in Section~4. Section~5 explores shuffle compatibility of those statistics related to the major index and descent number. In Section~6 we consider those statistics related to peaks. We conclude with a section outlining possible future directions and work.

\section{Permutation statistic definitions}

 Let $\pi\in L(U)$  be a permutation. Set $m=|U|$. The following is a list of all permutation statistics we will consider. We use the convention that when the name of a statistic is capitalized it is a set-valued statistic,  while lower case names are used for integer valued statistics.
\begin{enumerate}
\item[(i)] Recall that the \emph{descent set}, $\Des$ is defined by
$$\Des(\pi) = \{i  \mid i \text{ is a descent of } \pi\}\subseteq[m-1]$$
and the \emph{descent number} is the number of descents in the permutation, $\des(\pi) = \#\Des(\pi)$. An \emph{ascent} of a permutation is a position $i$ such that $\pi_i<\pi_{i+1}$. The set of the positions of ascents is denoted $\Asc(\pi)$, and $\asc(\pi)$ is the number of ascents. Two related permutations statistics are
\begin{align*}
\chi^-(\pi) &= \begin{cases}
1	&	\text{ if } 1\in \Des(\pi),\\
0	&	\text{ if } 1\not\in \Des(\pi)
\end{cases}\\
\chi^+(\pi) &= \begin{cases}
1	&	\text{ if } m-1\in \Asc(\pi),\\
0	&	\text{ if } m-1\not\in \Asc(\pi)
\end{cases}
\end{align*}
For example, if $\pi=685934$ then $\Des(\pi) = \{2,4\}$ since $8>5$ and $9>3$. Therefore $\chi^-(\pi) = 0$, but $\chi^+(\pi)=1$.

\item[(ii)]
Also as previously introduced, the \emph{major index} is given by
$$\maj(\pi) = \sum \limits_{i\in \Des(\pi)}i.$$

\item[(iii)]
A \emph{peak} of a permutation is a position $i$ such that $\pi_{i-1}<\pi_i>\pi_{i+1}$. The \emph{peak set} is 
$$\Pk(\pi) = \{ i \mid \pi_{i-1}<\pi_i>\pi_{i+1}  \}\subseteq [2,m-1].$$
and $\pk(\pi) = \#\Pk(\pi)$ is the \emph{peak number}.
A \emph{valley} of a permutation is a position $i$ such that  $\pi_{i-1}>\pi_i<\pi_{i+1}$. The \emph{valley set}, $\Val(\pi)$, and the valley number $\val(\pi)$ are defined analogously.  
Returning to our example, if $\pi=685934$  then $\Pk(\pi) = \{2,4\}$, $\pk = 2$, $\Val=\{3, 5\}$, and $\val = 2$.

\item[(iv)] A \emph{left peak} is a peak of the sequence $0\pi$, a \emph{right peak} is a peak of $\pi 0$, and an \emph{exterior peak} is a peak of $0 \pi 0$. The initial $0$ is added at position $0$ and the final $0$ is added at position $m+1$. We then have the \emph{left peak set}, $\Lpk$, \emph{left peak number} $\lpk$, the \emph{right peak set} $\Rpk$, the \emph{right peak number}, $\rpk$, the \emph{exterior peak set}, $\Epk$, and the \emph{exterior peak number}, $\epk$.
Continuing the previous example, $\Lpk(\pi) = \{2,4\}$, $\lpk(\pi)=2$, $\Rpk(\pi)=\Epk(\pi) = \{2,4, 6\}$, and $\rpk = \epk(\pi)=3$.

\item[(v)] A \emph{left valley} is a valley of $\infty\pi$, a \emph{right valley} is a valley of $\pi \infty$, and an \emph{exterior valley} is a valley of $\infty \pi \infty$. The definitions of the following statistics for valleys are analogous to those for peaks:
$$\Rval,\ \Lval,\ \Eval,\ \rval,\ \lval,\ \eval.$$
In our running example, $\Rval(\pi) =\{3,5\}$, $\rval(\pi) =2$,  $\Lval(\pi) = \Eval(\pi) = \{1,3,5\} $, and  $\eval(\pi)=\lval(\pi)= 3$.

\item[(vi)] 
A \emph{monotone factor} of a permutations is a factor that is either strictly increasing or strictly decreasing.
A \emph{birun} is a maximal  monotone factor. 
An \emph{updown} run is a birun of $0\pi$. The number of updown runs is denoted $\udr$.        
The number of biruns itself is not shuffle compatible, but it affords the most convenient definition of $\udr$ which is. As we will see in Section~6, one can also define $\udr$ using a linear combination of $\pk$, $\chi^+$, and $\chi^-$.
Using the usual example, $\udr(\pi) = 5$, where the 5 maximal monotone factors of $0\pi$ are $068$, $85$, $59$, $93$, and $34$.
\end{enumerate}

\section{A General Approach}

In this section we describe a method that is general enough to tackle most of the known shuffle compatible permutation statistics in a uniform and bijective manner. Let $\St$ be a descent statistic. In order to show it is shuffle-compatible bijectively we will use the following outline. 
\begin{enumerate}

\item[(i)] Reduce to showing only a special case of shuffle-compatibility using Corollary~\ref{cor-screduction} (b) or (c) below, whichever is most convenient. For the rest of this outline we assume (b) is chosen and let $m=|\pi|$.

\item[(ii)] Find a set $\Pi\subseteq L([m])$, called the set of {\em canonical permutations}, such that if $\pi,\pi'\in\Pi$ and the hypotheses of Definition \ref{defn-shufcompat}  are satisfied with $\sigma=\sigma'$, then clearly 
$\St(\pi\shuffle \sigma) = \St(\pi'\shuffle \sigma)$.

\item[(iii)] Find a function $d:L([m])\rightarrow\bbN$ 
such that for any $\pi\not\in\Pi$  there is a $\pi'\in L([m])$ with $\St(\pi')=\St(\pi)$ and $d(\pi')<d(\pi)$ as well as an $\St$-preserving bijection $\pi\shuffle \sigma \to \pi'\shuffle \sigma$.

\end{enumerate}

To see that this suffices to show shuffle-compatibility, repeatedly apply step (iii) to generate a sequence of permutations
with decreasing values of $d$ and corresponding $\St$-preserving bijections.  This can only be done a finite number of times since the range of $d$ is $\bbP$ which is well ordered.  Upon termination we must have $\pi'\in\Pi$ with $\St(\pi')=\St(\pi)$ and, via composition, an $\St$-preserving bijection $\pi\shuffle \sigma \to \pi'\shuffle \sigma$ where $\pi$ is the permutation we started with.  By step (ii), this is enough to prove shuffle compatibility.  We also note that often the bijections in step (iii) will be constructed using (variations of) a map which we will call the fundamental bijection, see  Definition~\ref{defn-fundbij}.

The next lemma gives our first reduction. It is at the heart of the proof of Corollary~\ref{cor-screduction} which is our main tool for reducing the number of cases under consideration. An example of its  proof is given afterwards.
\begin{lem}
\label{lem-setcompat} 
 Let $\St$ be a descent statistic, and consider four permutations $ \pi, \pi', \sigma, \sigma'$ such that $\dom(\pi)\cap \dom(\sigma)=\dom(\pi')\cap \dom(\sigma')=\emptyset$. If $\std \pi=\std\pi'$ and $\std \sigma=\std \sigma'$ then
$$\St(\pi\shuffle \sigma) = \St(\pi'\shuffle \sigma').$$
\end{lem}

\begin{proof}
Our method of proof will reflect the philosophy of our general approach, but with some modifications since we are only showing a special case of shuffle compatibility and do not yet have the full power of Corollary~\ref{cor-screduction}. In place of (i) above, we reduce the possible domains of our permutations by observing that since permutation statistics only depend on the relative order we may assume without loss of generality that 
\beq
\label{domains}
\dom(\pi)\sqcup \dom(\sigma) = \dom(\pi')\sqcup \dom(\sigma') = [m+n]
\eeq
where $m=|\pi| = |\pi'|$, and $n = |\sigma|=|\sigma'|$. Let $U=[m]$ and $V=[n]+m$.

To mimic (ii), we consider the set
$$\Pi = L(U)\times L(V).$$
For suppose we have $(\pi,\si),(\pi',\si')\in \Pi$ satisfying the hypotheses of the Lemma.  Then
$\pi,\pi'\in L([m])$ implies 
$$
\pi=\std\pi=\std\pi'=\pi'.
$$
Similarly $\sigma=\sigma'$.  So clearly $\St(\pi\shuffle \sigma) = \St(\pi'\shuffle \sigma')$.

For (iii), we assume only~\ree{domains}, and produce an $\St$-preserving  bijection 
$$
\pi\shuffle \sigma \to \std_U(\pi) \shuffle \std_V(\sigma).
$$
Our measure of how close a pair of permutations is to being in $\Pi$  is given by $\#O$ where
$$O= \{(i,j)\in \dom(\pi)\times \dom(\sigma)\ | \ i>j \}.$$
A pair will have $\#O=0$ exactly when $(\pi, \sigma)\in \Pi$.

Now if $(\pi,\sigma)\not\in\Pi$ we will produce a pair of permutations $(\pi'', \sigma'')$ with $\St(\pi)=\St(\pi'')$, $\St(\sigma)=\St(\sigma'')$ and a $\St$-preserving bijection $\pi\shuffle \sigma \to \pi'' \shuffle \sigma''$ that reduces $\#O$. This suffices because repeatedly applying this operation will produce the pair of permutations $(\std_U(\pi), \std_V(\sigma))\in\Pi$, and (by composition) the desired  $\St$-preserving bijection. An analogous argument gives a $\St$-preserving bijection $\pi'\shuffle \sigma' \to \std_U(\pi')\shuffle \std_V(\sigma')$. 
Then we will have, using the argument in the paragraph about part (ii),
$$\St(\pi\shuffle \sigma) = \St(\std_U(\pi)\shuffle \std_V(\sigma)) =  \St(\std_U(\pi')\shuffle \std_V(\sigma'))=
 \St(\pi'\shuffle \sigma')$$
as required. 

We now construct $(\pi'', \sigma'')$ and the $\St$-preserving bijection. Since $(\pi,\sigma)\not\in\Pi$, there exists a pair $(i,i-1)\in O$ such that $i\in\dom(\pi)$ and $i-1 \in \dom(\sigma)$. Set $\pi'' = (i,i-1)\pi$ and $\sigma'' = (i,i-1)\sigma$ where $(i,i-1)\pi$ is the permutation $\pi$ with $i$ replaced by $i-1$ and similarly for $(i,i-1)\sigma$. Let $\tau\in \pi\shuffle \sigma$. Then the bijection is given by
$$
T_i(\tau) = \begin{cases}
(i,i-1)\tau & \text{ if $i$, $i-1$ are not adjacent in $\tau$},\\
\tau & \text{ otherwise},
\end{cases}
$$
where $(i,i-1)\tau$  is $\tau$ with $i$ and $i-1$ interchanged.

This map is its own inverse, hence a bijection. To see that the image of the map is in $\pi''\shuffle \sigma''$ note that if $i, i-1$ are not adjacent then $T_i(\tau)\in \pi''\shuffle \sigma''$ since $T_i(\tau)$ is the unique shuffle of $\pi''$ and $\sigma''$ whose word satisfies $\omega(T_i(\tau))=\omega(\tau)$. And if  $i$ and $i-1$ are adjacent in $\tau$, then $\tau$ is easily seen to also be a shuffle of $\pi''$ and $ \sigma''$. 

The map $T_i$ is $\Des$ preserving because swapping the positions of $i,i-1$ when $i$ and $i-1$ are not adjacent will not change the order relation between any adjacent pairs. Indeed, given any $j\not\in \{i-1,i\}$ that is adjacent to one or both of $i,i-1$, it is clear that either both $j>i$ and $j>i-1$, or $j<i-1$ and $j<i$, and hence the inequalities are preserved when interchanging $i, i-1$. It follows from the definition of a descent statistic that $T_i$ is also $\St$ preserving. 
\end{proof}

As an example, let $\boldsymbol{U=\{\boldsymbol{1},\boldsymbol{2},\boldsymbol{4}\}}$, $V=\{3,7\}$ and $\boldsymbol{\pi=241}$ and $\sigma=73$. Then 
$$
\pi\shuffle \sigma = \{\boldsymbol{241}73,\ \boldsymbol{24}7\boldsymbol{1}3,\ \boldsymbol{24}73\boldsymbol{1},\  \boldsymbol{2}7\boldsymbol{41}3,\ \boldsymbol{2}7\boldsymbol{4}3\boldsymbol{1},\ \boldsymbol{2}73\boldsymbol{41},\ 7\boldsymbol{241}3,\ 7\boldsymbol{24}3\boldsymbol{1},\ 7\boldsymbol{2}3\boldsymbol{41},\ 73\boldsymbol{241} \}.
$$
Taking $\St$ to be the peak set statistic, 
$$
\Pk(\pi\shuffle \sigma) = \{\{\{2\}^2,\ \{3\}^4,\ \{4\}^2,\  \{2,4\}^2\}\}.
$$

Now standardize to $[m+n]$ by replacing $\sigma$ with $\widetilde{\sigma}=53$, and $V$ with $\widetilde{V}=\{3,5\}$ to obtain
$$
\pi\shuffle \widetilde{\sigma} = \{\boldsymbol{241}53,\ \boldsymbol{24}5\boldsymbol{1}3,\ \boldsymbol{24}53\boldsymbol{1},\  \boldsymbol{2}5\boldsymbol{41}3,\ \boldsymbol{2}5\boldsymbol{4}3\boldsymbol{1},\ \boldsymbol{2}53\boldsymbol{41},\ 5\boldsymbol{241}3,\ 5\boldsymbol{24}3\boldsymbol{1},\ 5\boldsymbol{2}3\boldsymbol{41},\ 53\boldsymbol{241} \}.
$$
Clearly $\Pk(\pi\shuffle\widetilde{\sigma}) = \Pk(\pi\shuffle \sigma)$.
We next would like to change $\boldsymbol{U}$ to $\boldsymbol{U'=\{1,2,3\}}$ and $\widetilde{V}$ to $\widetilde{V}'=\{4,5\}$. This can be done using $(4,3)\in O$.
We apply 
$$T_4(\pi\shuffle \widetilde{\sigma}) = \{\boldsymbol{231}54, \boldsymbol{23}5\boldsymbol{1}4, \boldsymbol{23}54\boldsymbol{1},  \boldsymbol{2}5\boldsymbol{31}4 ,\boldsymbol{2}54\boldsymbol{31}, \boldsymbol{2}5\boldsymbol{3}4\boldsymbol{1}, 5\boldsymbol{231}4, 5\boldsymbol{2}4\boldsymbol{3}\boldsymbol{1}, 5\boldsymbol{2}\boldsymbol{3}4\boldsymbol{1},54\boldsymbol{231} \}$$
where, for instance, $T_4(5\boldsymbol{241}3) = 5\boldsymbol{231}4$ since $3$ and $4$ are not adjacent. On the other hand, $T_4(5\boldsymbol{2}3\boldsymbol{41})=5\boldsymbol{23}4\boldsymbol{1}$ since $3,4$ are adjacent. One can check that  the distribution with respect to $\Pk$ remains unchanged.

The following corollary shows that in order to check shuffle compatibility, it suffices to check the special case when the domains of the permutation have some fixed relation with each other. This reduction greatly simplifies the required arguments for showing statistics are shuffle compatible.

\begin{cor}
\label{cor-screduction}
Suppose $\St$ is a descent statistic. The following are equivalent.
\begin{enumerate}
\item[(1)] The statistic $\St$ is shuffle compatible.
\item[(2)]
If $\St(\pi)=\St(\pi')$ where $\pi,\pi'\in L([m])$, and $\sigma\in L([n]+m)$ for some $m,n\geq 0$, then
 $$\St(\pi\shuffle \sigma)=\St(\pi'\shuffle \sigma).$$
\item[(3)]  
If $\St(\sigma)=\St(\sigma')$ where  $\sigma,\sigma' \in L([n]+m)$, and $\pi\in L([m])$ for some $m,n\geq 0$,
then $$\St(\pi\shuffle \sigma')=\St(\pi\shuffle \sigma).$$
\end{enumerate}
\end{cor}

\begin{proof}
It is clear that (1) implies (2) and (3) as these are special cases of the definition of shuffle compatibility. Assume (2) holds and let $\pi, \pi'$ be any two permutations of the same length $m$ such that $\St(\pi)=\St(\pi')$. Let $\sigma, \sigma'$ be two permutations of the same length $n$ and disjoint from $\pi, \pi'$ (respectively) such that $\St(\sigma)=\St(\sigma')$. Set $U=[m]$, $U^+=[m]+n$, $V =[n]$, and $V^+ = [n]+m$. Then by Lemma \ref{lem-setcompat} and our assumption,
$$
\begin{array}{rllll}
\St(\pi\shuffle \sigma) &= & \St(\std_{U}(\pi)\shuffle \std_{V^+}(\sigma))\hspace{.5cm} &&\text{ by Lemma \ref{lem-setcompat}}\\
&= &\St(\std_{U}(\pi')\shuffle \std_{V^+}(\sigma))&& \text{ by (2)} \\
&=& \St(\std_{U^+}(\pi')\shuffle\std_{V}(\sigma) ) &&\text{ by Lemma \ref{lem-setcompat}} \\
&=&  \St(\std_{U^+}(\pi')\shuffle \std_{V}(\sigma'))  &&\text{ by (2)}\\
&=& \St(\pi'\shuffle \sigma')  &&\text{ by Lemma \ref{lem-setcompat}.}
\end{array}$$
The proof that (3) implies (1) is very similar.
\end{proof}

\section{Set valued statistics}

Our first main results are to give bijective proofs for the shuffle compatibility of some set valued statistics.  The statistic $\Des$ was given a different bijective proof in \cite{GesZhu17}, so the novelty here is the bijective proofs of the remaining set valued statistics as well as the uniform manner in which they are attained.

\begin{defn}
\label{defn-fundbij}
Given permutations $\pi,\sigma$ with disjoint domains and a third permutation $\pi'$ dijoint from $\sigma$ with $|\pi|=|\pi'|$, define the \emph{fundamental bijection} 
$$\Phi:\pi\shuffle \sigma \to \pi' \shuffle \sigma$$
by
$$\Phi(\tau) = \tau'$$
where $\tau'\in \pi'\shuffle \sigma$ is the unique permutation with $\omega(\tau')=\omega(\tau)$. This amounts to replacing the elements of $\pi$ with the elements of $\pi'$ in the same order and positions as in $\tau$.  If  instead one holds $\pi$ fixed and replaces $\sigma$ with $\sigma'$ then one obtains a bijection which we call $\widetilde{\Phi}$.\hfill $\square$
\end{defn}

For example, if $\pi=132$, $\sigma=4589$, $\tau=1453829\in \pi\shuffle \sigma$ and $\pi' = 361$, then $\Phi(\tau) =3456819$.

The following theorem establishes the shuffle compatibility of some set valued statistics. These were initially proven by Gessel, Zhuang and Gringberg in \cite{GesZhu17} and \cite{Grin2018} by lengthier and  primarily algebraic methods. An advantage of our approach is the directness and uniformity with which the results are obtained.

\begin{thm}
\label{thm-shufcompatsetd}
The set valued statistics $\Des$, $\Pk$, $\Lpk$, $\Rpk$, and $\Epk$ are all shuffle compatible.
\end{thm}

\begin{proof}
By Corollary~\ref{cor-screduction}, part (2), we may reduce to showing that for $\pi,\pi'\in L([m])$ and $\sigma \in L([n]+m)$, if  we have $\St(\pi)=\St(\pi')$ then it follows that $\St(\pi\shuffle \sigma) = \St(\pi'\shuffle \sigma)$ for the five statistics listed above. 

The fundamental bijection $\Phi:\pi\shuffle \sigma \to \pi' \shuffle \sigma$ will be the map we will use for all five statistics.  Because these cases are so straightforward, we will not have to find a canonical set of permutations.

For each of these statistics, $\St\in \{\Des, \Pk, \Lpk, \Rpk, \Epk\}$, we will give a complete list of the cases that determine whether a given position will contribute to $\St(\tau)$ for a shuffle $\tau\in \pi \shuffle \sigma$. It will then be easy to check  that $\St$ will be  preserved by $\Phi$.

\vspace{.2cm}
\begin{itemize}
\item \textbf{Descent set, $\Des$}:\\
Observe that in a shuffle $\tau\in \pi\shuffle \sigma$ we have $i\in\Des\tau$ if and only if $\tau_i\tau_{i+1}$ equals one of 
\begin{enumerate}
\item $\pi_j\pi_{j+1}$ where $j\in\Des\pi$,
\item $\sigma_k\sigma_{k+1}$ where $k\in\Des\sigma$, or
\item $\sigma_k\pi_j$.
\end{enumerate}
It is now easy to check that the descent set is preserved in passing from $\tau$ to $\Phi(\tau)$.

\vspace{.2cm}
\item \textbf{Peak set, $\Pk$}:\\
For a shuffle $\tau\in \pi\shuffle \sigma$, we have $i\in \Pk\tau$ if and only if $\tau_{i-1}\tau_{i}\tau_{i+1}$ equals one of 
\begin{enumerate}
\item $\pi_{j-1}\pi_{j}\pi_{j+1}$ where $j\in \Pk \pi$,
\item $\sigma_{k-1}\sigma_{k}\sigma_{k+1}$ where $k\in \Pk \sigma$,
\item $\sigma_{k}\sigma_{k+1}\pi_j$ where $k\in \Asc \sigma$,
\item $\pi_j\sigma_{k}\sigma_{k+1}$ where $k\in \Des \sigma$,
\item $\pi_j\sigma_{k}\pi_{j+1}$.
\end{enumerate}

This makes it simple to check that a peak in $\tau$ will remain one in $\Phi(\tau)$. For example, in case 3 we have $\sigma_{k}<\sigma_{k+1}>\pi_j$ in $\tau$ so that the position of $\sigma_{k+1}$ is a peak of $\tau$. Upon replacing $\pi$ with $\pi'$ we have $\sigma_{k}<\sigma_{k+1}>\pi_j'$ since every element of  $\pi'$ is less than every element of $\sigma$. Therefore the position of $\sigma_{k+1}$ is a peak in $\Phi(\tau)$ at the same position as it was in $\tau$. 
Using similar arguments and $\Phi^{-1}$, one sees that a position that is a peak of $\Phi(\tau)$ must also be a peak of $\tau$ and so $\Phi$ is peak preserving.

\vspace{.2cm}
\item \textbf{Left peak set, $\Lpk$}:\\
For a shuffle $\tau\in \pi\shuffle \sigma$, note that we have $\Pk(\tau)\subseteq \Lpk(\tau)$ so that the above cases for the peak set show that for $i\geq 2$ we have  $i\in \Lpk(\tau)$ if and only if $i\in \Lpk(\Phi(\tau))$.
It therefore remains only to check what happens when $i=1$. But $1\in \Lpk(\tau)$ if and only if $0\tau_1\tau_2$ equals one of
\begin{enumerate}
\item $0\pi_{1}\pi_{2}$ where $1\in \Lpk(\pi)$
\item $0\sigma_1\sigma_2$ where $1\in \Lpk(\sigma)$
\item $0\sigma_1\pi_1$
\end{enumerate}

The check that left peaks at $i=1$ are preserved is similar to the arguments for $\Pk$.  So here and for the following statistics we have left this verification to the reader.

\vspace{.2cm}
\item \textbf{Right peak set, $\Rpk$}:\\
The argument is analogous to that of $\Lpk$, except that we now need additional cases at the right end of $\tau$. 
Note that $m+n\in \Rpk(\tau)$ if and only if $\tau_{m+n-1}\tau_{m+n}0$ equals one of
\begin{enumerate}
\item $\pi_{m-1}\pi_{m}0$ where $m\in \Rpk(\pi)$
\item $\sigma_{n-1}\sigma_n 0$ where $n\in \Rpk(\sigma)$
\item $\pi_m\sigma_n 0$
\end{enumerate}

\item \textbf{External peak set, $\Epk$}:\\
Since $\Epk(\tau) = \Lpk(\tau)\cup\Rpk(\tau)$ and the single bijection $\Phi$ preserves both $\Rpk$ and $\Lpk$, we have that $\Epk$ 
is also preserved under $\Phi$.
\end{itemize}

This completes the proof of the shuffle compatibility of these five statistics.\end{proof}

A couple of observations are appropriate here. First, note that if there is a single bijection which shows the shuffle compatibility of two or more permutation statistics, then it follows immediately that any tuple of these statistics is also shuffle compatible. For example, since $\Lpk$ and $\Rpk$ are both shuffle compatible by means of the bijection $\Phi$, then so  is $(\Lpk, \Rpk)$. Therefore any tuple from the above five statistics is also shuffle compatible. This gives one answer to a question of Gessel and Zhuang in \cite{GesZhu17} as to when a tuple of statistics is shuffle compatible. Second, some statistics determine others. For example $(\Des,\Pk)$ is shuffle compatible since $\Des$ is shuffle compatible and  $\Pk$ can be determined from $\Des$. 

\begin{thm}
\label{thm-shufcompatseta}
The statistics $\Asc$, $\Val$, $\Lval$, $\Rval$, and $\Eval$ are shuffle compatible.
\end{thm}

\begin{proof}
This proof closely parallels that of the previous theorem except we use 
part (3) of Corollary~\ref{cor-screduction} in place of part (2),
$\widetilde{\Phi}$ in place of $\Phi$, and $\infty$ in place of $0$.
Because of the similarity, we only indicate how to do $\Asc$.

\vspace{.2cm}

Observe that in a shuffle, $\tau\in \pi\shuffle \sigma$, we have $i\in\Asc\tau$ if and only if $\tau_i\tau_{i+1}$ equals one of 
\begin{enumerate}
\item $\pi_j\pi_{j+1}$ where $j\in\Asc \pi$,
\item $\sigma_k\sigma_{k+1}$ where $k\in\Asc\sigma$, or
\item $\pi_j\sigma_k$.
\end{enumerate}
It is now easy to check that the ascent set is preserved in passing from $\tau$ to $\widetilde{\Phi}(\tau)$.
\end{proof}

\section{The major index}

For the next proof, we need to introduce a labeling on the spaces of a permutation. Let $\pi$ be a permutation of length $m$ with $\des(\pi)=k$.  Then by a {\em space} of $\pi$ we mean the gap between two adjacent elements of $\pi$. There is, by convention, an initial  space before the first element of $\pi$ and a  final space after the last element of the permutation.  Label these spaces by assigning the right-most (final) space the label $0$ then labeling the spaces after descents of $\pi$ with the integers in $[k]$ from right to left, then labeling the remaining spaces with the integers in $[k+1,m]$ from left to right. 
Equivalently, we label the spaces of $\pi$ corresponding to descents of $0\pi0$ from right to left, and then the spaces of $\pi$ corresponding to ascents of $0\pi0$ from left to right using the elements of $[0,m]$.  In what follows we make no distinction between a space and its label.
For example if $\pi=265781$ then the labeled permutation is
$${}^3 2^4 6^2 5^5 7^6 8^1 1^0$$ 
with the raised numbers being the labels of the spaces. 
If $\pi_i$ and $\pi_{i+1}$ are the elements on either side of space $x$ then we say there is a {\em descent} or {\em ascent at space $x$} if $i\in\Des(\pi)$ or $i\in\Asc(\pi)$, respectively.

It is well known that inserting a number greater than $\max \dom(\pi)$ in space $i$ increases $\maj\pi$ by $i$.  Continuing our example, inserting $9$ in space $4$ of $\pi$ gives the permutation $2965781$ with $\maj(2965781)=11=\maj(265781) +4$. This fact 
is used in one of the standard proofs that the generating function for $\maj$ over the permutations of $[n]$ is $[n]_q!$. We will now see that this is a crucial tool for proving certain shuffle compatibility results.

\begin{thm}
\label{thm-desshufcompat}
The permutation statistics $\des$ and $(\maj,\des)$  are shuffle compatible.
\end{thm}

\begin{proof}
Our first step is to use Corollary~\ref{cor-screduction} part (3) to reduce to showing that $\pi \in L(m)$ and $\sigma, \sigma'\in L([n]+m)$ with $\St(\sigma)=\St(\sigma')$ implies $\St(\pi\shuffle \sigma) = \St(\pi \shuffle \sigma')$ for each $\St\in \{\des, (\maj,\des)\}$.
The core of the proof is the existence of certain bijections that preserve $\des$, lower $\maj$ by one, and allow us to replace $\sigma$ with a permutation that is closer to being in our chosen set of canonical permutations, as outlined in the general approach.
For the permutations with $\des\sigma=p$ we will use the canonical set
$$\Pi = \{\sigma\in L([n]+m) \mid \Des(\sigma)=[p]\} $$ 
which consists of the permutations with a sequence of $p$ descents followed by a sequence of ascents. Given two permutations $\sigma,\sigma'\in \Pi$, we know by Theorem~\ref{thm-shufcompatsetd} that $\Des(\pi\shuffle \sigma)=\Des(\pi \shuffle \sigma')$ and hence the same holds for any descent statistic. This shows that part (ii) of the general approach is satisfied.

Our measure of how close a permutation is to being in $\Pi$ is $d: L([n]+m)\to \bbN$ given by $d(\sigma) =\maj(\sigma)$.   Note that among all permutations in $L([n]+m)$ with $\des\sigma=p$, those in $\Pi$ have the minimum possible $\maj$, namely $\binom{p+1}{2}$.  Our strategy will be to find a bijection between shuffles $\tau\in\pi\shuffle\sigma$ and the shuffles of $\pi$ with an element of $\Pi$  which preserves $\des$ and lowers the major index of each $\tau$  by the same amount, namely $\maj(\sigma)-\binom{p+1}{2}$.  This will prove the theorem.

To reduce a permutation to one in $\Pi$ we will move their descents to the left one position at a time. More specifically, if $\sigma \not\in \Pi$ then there is at least one position $i\geq 2$ such that $\sigma_{i-1}<\sigma_i>\sigma_{i+1}$. Let $\sigma''$ be any permutation such that 
$$\Des(\sigma'') = \left(\Des(\sigma)\setminus \{i\}\right)\cup \{i-1\}.$$
Note that this preserves $\des$ but lowers $\maj$ by one in passing from $\sigma$ to $\sigma''$ . The bijection we will define between $\pi \shuffle \sigma$ and $\pi \shuffle \sigma''$ will  have the same properties and so, by iteration, complete the proof.

For $\tau \in \pi\shuffle \sigma$, write  $\tau$ as a concatenation $\tau=\tau^a \tau^b \tau^c$ where $\tau^b$ is the factor of $\tau$ between
but not including $\sigma_{i-1}$ and $\sigma_{i+1}$. Then $\tau^a$ and $\tau^c$ are the remaining initial and final factors of $\tau$, respectively. Note that there is exactly one element of $\sigma$ in $\tau^b$ and that it is larger than all the elements of $\pi$. Consider the permutation $\delta$ that is $\tau^b$ with $\sigma_i$ removed. All spaces will be spaces of $\delta$. Let $x$ be the space of $\delta$ from which $\sigma_i$ was removed from and set $(\tau^{b})''$ to be the permutation $\delta$ with $\sigma_i$ inserted into the space $x-1$ where $x-1$ is taken modulo  $|\delta|+1$.

Define a map 
$\Theta: \pi \shuffle \sigma \to \pi \shuffle \sigma''$
by
 $$\Theta(\tau)=\tau''$$
 where $\tau''$ is the unique element of $\pi \shuffle \sigma''$ such that $\omega(\tau'') = \omega(\tau^a(\tau^b)''\tau^c)$.

We now show that $\Theta$ has the desired properties, namely that it is $\des$ preserving and satisfies $\maj(\Theta(\tau)) = \maj(\tau)-1$. There are two cases to check, based on the label $x$ of the original  space that $\sigma_i$ occupied.

\begin{enumerate}
\item  If $1\leq x\leq \des(\delta)+1$ :\\
First note that in this case $\delta$ cannot be empty since $x$ exists in this range.
Let $\tau_j$ be the element of $\tau$ directly before $\sigma_i$ and let $\tau_k$ be the be the element of $\tau$ directly before  space $x-1$ of $\delta$.

The map $\Theta$ removes $\sigma_i$ from the position after $\tau_j$ and inserts $\sigma_i''$ in the position directly after $\tau_k$ while changing the order relation from $\sigma_{i-1}<\sigma_{i}>\sigma_{i+1}$ to $\sigma_{i-1}''>\sigma_{i}''<\sigma_{i+1}''$. There are no descents, $l$, in $\tau$ with $j+2\leq l \leq k-1$. Therefore we only have to analyze what happens  at positions $j,\ j+1,\ k-1$, and $k$. First, assume that $j+2 \not=k$.

An illustration of this case in a generic small example is shown in Figure~\ref{fig-desmaj1}. The left picture is an example of the initial state of $\tau^b$ and the right picture is of the resulting part of $\tau''$ after applying $\Theta$. Each node represents an element of $\tau$ or $\tau''$ and the lines connecting them represent the order relation between adjacent elements. For example, a line with positive slope corresponds to the first element being smaller than the second. 

\begin{itemize}
\item In $\tau$:
	\begin{itemize}
	\item $j\not\in \Des(\tau)$:\\
	The position $j$ is never a descent of $\tau$ since $\tau_{j+1}=\sigma_i$, and $\tau_j$ is either $\sigma_{i-1}$  or  in $\pi$.  In both cases $\tau_j<\tau_{j+1}$.
	\vspace{.2cm}
	
	\item $j+1\in \Des(\tau)$:\\
		The position $j+1$ is always a descent of $\tau$ since $\tau_{j+1} = \sigma_i$ 
		and $\tau_{j+2}\in \pi$ because of the range of $x$.
		\vspace{.2cm}
	\item $k-1\not \in \Des(\tau)$.\\
		Since $j+2\not=k$,  the definition of the space labeling and the range of $x$ show that $k-1$ is an ascent of $\tau$.		\vspace{.2cm}

	\item $k\in \Des(\tau)$ if and only if $x\neq 1$.\\
		If $x=1$, then $\tau_k\in \pi$ and $\tau_{k+1} = \sigma_{i+1}$ and so $k$ is an ascent of $\tau$. On the other hand, if $x\not=1$, then $k$ is a descent of $\tau$ by the definition of the space labeling and the range of $x$.
\vspace{.2cm}

	\end{itemize}
\item In $\tau''$: 
	\begin{itemize}
	\item  $j\in \Des(\tau'')$:\\
	The position $j$ is always a descent of $\tau''$ since $\tau_{j+1}''\in \pi$,  and either $\tau_j''=\sigma_{i-1}''$ or $\tau_j''\in\pi$ with $j$  corresponding to the descent at space $x$ of $\delta$.
			\vspace{.2cm}
			
	\item $j+1\not \in \Des(\tau'')$:\\
	The position $j+1$ is never a descent of $\tau''$ since  $\tau_{j+1}''\in \pi$, and either $\tau_{j+2}''=\sigma_i''$ or $\tau_{j+2}''\in\pi$ with $j+1$ corresponding to an ascent of $\delta$ by the definition of the space labeling and the range of $x$.
			\vspace{.2cm}
			
	\item $k-1 \not\in \Des(\tau'')$ \\
		Since $j+2\not=k$ we have $\tau_{k-1}''\in\pi$ and $\tau_k''=\sigma_i''$ so that $k-1$ is an ascent.
	\vspace{.2cm}
	
	\item $k\in \Des(\tau'')$ if and only if $x\neq 1$.\\
	We have $\tau_{k}''=\sigma_i''$. If $x=1$ then $\tau_{k+1}''=\sigma_{i+1}''$ and hence $k\not\in \Des(\tau'')$ by the choice of $\sigma''$. On the other hand if $x\not=1$, then $\tau_{k+1}''\in \pi$ and hence $\sigma_i''>\tau_{k+1}''$.

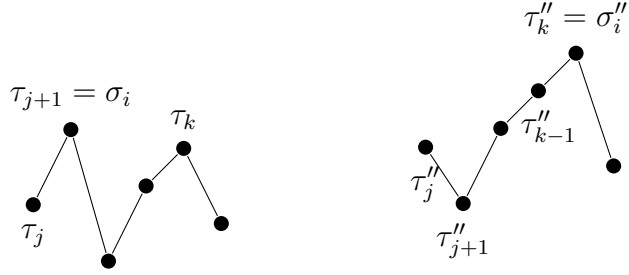
\begin{figure}
\label{fig-desmaj1}
\centering
\begin{tikzpicture}
[blackdot/.style={circle, fill=black,thick, inner sep=2pt, minimum size=0.2cm}, scale=.5] 
     \tikzstyle{edge_style} = [draw=black]
     \node (v1) at (-3,0.5) [blackdot,label=below:$\tau_j$]{};
     \node (v2) at (-2,2.5) [blackdot, label=above:{$\tau_{j+1}=\sigma_i$}]{};
     \node (v3) at (-1,-1) [blackdot]{};
     \node (v4) at (0,1) [blackdot]{};
     \node (v5) at (1,2) [blackdot, label=above:$\tau_k$]{};
     \node (v6) at (2,0) [blackdot]{};
     \draw[edge_style] (v1) to (v2);
     \draw[edge_style] (v2) to (v3);
     \draw[edge_style] (v3) to (v4);
     \draw[edge_style] (v4) to (v5);
     \draw[edge_style] (v5) to (v6);
\end{tikzpicture} \hspace{2cm}
\begin{tikzpicture}
[blackdot/.style={circle, fill=black,thick, inner sep=2pt, minimum size=0.2cm}, scale=.5] 
     \tikzstyle{edge_style} = [draw=black]
     \node (v1) at (-2,0.5) [blackdot,label=below:$\tau_j''$]{};
     \node (v2) at (-1,-1) [blackdot, label=below:$\tau_{j+1}''$]{};
     \node (v3) at (0,1) [blackdot]{};
     \node (v4) at (1,2) [blackdot, label=below:$\rule{8pt}{0pt}\tau_{k-1}''$]{};
     \node (v5) at (2,3) [blackdot, label=above:{$\tau_{k}'' = \sigma_i''$}]{};
     \node (v6) at (3,0) [blackdot]{};
     \draw[edge_style] (v1) to (v2);
     \draw[edge_style] (v2) to (v3);
     \draw[edge_style] (v3) to (v4);
     \draw[edge_style] (v4) to (v5);
     \draw[edge_style] (v5) to (v6);
\end{tikzpicture}
\caption{A schematic drawing for the case when $1\leq x\leq \des(\delta)+1$ and $j+2\not=k$.}
\end{figure}

	\end{itemize}
\end{itemize} 

When $j+2=k$, we have the same two lists but with the item concerning $k-1$ removed since $k-1=j+1$ and so the item concerning $j+1$ covers this case.

Comparing the two lists, we have
\begin{equation}
\label{tau''}
\Des(\tau'') = \left(\Des(\tau)\setminus \{j+1\}\right)\cup \{j\}
\end{equation}
and, in particular, $\des(\tau'')=\des(\tau)$. This relation between descent sets also makes it clear that $\maj(\tau'') = \maj(\tau)-1$.

\item If $\des(\delta)+2\leq x \leq |\delta|+1$  modulo $|\delta|+1$:\\
Define $\tau_j$ and $\tau_k$ as before.

Again, the map $\Theta$  removes $\sigma_i$ from the position after $\tau_j$ and inserts $\sigma_i''$ in the position directly after $\tau_k$ while changing the order relation from $\sigma_{i-1}<\sigma_{i}>\sigma_{i+1}$ to $\sigma_{i-1}''>\sigma_{i}''<\sigma_{i+1}''$. Note however, that it is now possible for $\delta$ to be empty.
In that case $\Theta$ just changes the peak at $\sigma_i$ at position $j+1$ in $\tau$ to a valley in $\tau''$.  So clearly equation~(\ref{tau''}) still holds.

Therefore assume $\delta\neq \emptyset$ and note that all positions strictly between $k$ and $j$ will be descents in $\tau$.  It is easy to see that these positions will remain descents after applying $\Theta$. 
So we only need to check what happens at positions $ j, j+1$ and $k$. 

An illustration of this case is given in Figure \ref{fig-desmaj2}.  The left picture is an illustration of an example of the initial state of $\tau^b$ and the right picture is of the resulting part $\tau''$ after applying $\Theta$.

\begin{itemize}
\item In $\tau$:
	\begin{itemize}
	\item $j\not\in \Des(\tau)$:\\
	The position $j$ is never a descent of $\tau$ since  $\tau_j \in  \pi$ and $\tau_{j+1}=\sigma_i$ by definition.
	\vspace{.2cm}
	
	\item $j+1\in \Des(\tau)$:\\
		The position $j+1$ is always a descent of $\tau$ since $\tau_{j+1} = \sigma_i$ 
		and $\tau_{j+2}$ is $\sigma_{i+1}$ or is in $\pi$. In both cases $\tau_{j+1}>\tau_{j+2}$.
		\vspace{.2cm}
		
	\item $k \in \Des(\tau)$ if and only if $x= \des(\delta)+2$.\\
	If $x=\des(\delta)+2$, then $x-1=\des(\delta)+1$ is the the space before $\delta$. Hence $\tau_k=\sigma_{i-1}$ and $\tau_{k+1}\in \pi$, so $k \in \Des(\tau)$. On the other hand, if $x\not=\des(\delta)+2$, then $k$ is an ascent of $\delta$ and hence $\tau$ by the definition of the space labeling and the range of $x$.
	
	\end{itemize}
	
\item In $\tau''$:
	\begin{itemize}	
	\item $j\in \Des(\tau'')$:\\
	The position $j$ is always a descent of $\tau''$ since $\tau_{j+1}''\in\pi$, and either $\tau_{j}''=\sigma_{i}''$ or $\tau_{j}''\in\pi$ and $j$ corresponds to the descent of $\delta$ at the space previous to $x$.
	\vspace{.2cm}
				
	\item $j+1\not \in \Des(\tau'')$:\\
The position $j$ is never a descent of $\tau''$ since $\tau_{j+1}''\in\pi$, and either $\tau_{j+2}''=\sigma_{i+1}''$ or $\tau_{j+2}''\in\pi$ and $j+1$ corresponds to 
the ascent of $\delta$ at space $x$.
			\vspace{.2cm}
			
	\item $k \in \Des(\tau'')$ if and only if $x=\des(\delta)+2$.\\	
	If $x=\des(\delta)+2$, then $x-1=\des(\delta)+1$ is the the space before $\delta$. Hence $\tau_k''=\sigma_{i-1}''$ and $\tau_{k+1}''=\sigma_i$, so $k \in \Des(\tau)$. On the other hand,  if $x\not=\des(\delta)+2$ then $\tau_{k}''\in \pi$ and $\tau_{k+1}''=\sigma_i''$, so that $k\not \in \Des(\tau'')$.
	
	\end{itemize}
\end{itemize} 

Thus in this case as well equation~(\ref{tau''}) continues to hold.
\end{enumerate}

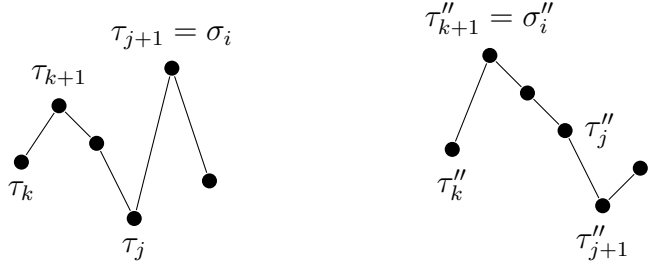
\begin{figure}

\centering
\begin{tikzpicture}
[blackdot/.style={circle, fill=black,thick, inner sep=2pt, minimum size=0.2cm}, scale=.5] 
     \tikzstyle{edge_style} = [draw=black]
     \node (v1) at (-3,0.5) [blackdot,label=below:$\tau_k$]{};
     \node (v2) at (-2,2) [blackdot, label=above:{$\tau_{k+1}$}]{};
      \node (v3) at (-1,1) [blackdot]{};
     \node (v4) at (0,-1) [blackdot, label=below:{$\tau_{j}$}]{};
     \node (v5) at (1,3) [blackdot, label=above:{$\tau_{j+1}=\sigma_i$}]{};
     \node (v6) at (2,0) [blackdot]{};
     \draw[edge_style] (v1) to (v2);
     \draw[edge_style] (v2) to (v3);
     \draw[edge_style] (v3) to (v4);
     \draw[edge_style] (v4) to (v5);
     \draw[edge_style] (v5) to (v6);
\end{tikzpicture} \hspace{2cm}
\begin{tikzpicture}
[blackdot/.style={circle, fill=black,thick, inner sep=2pt, minimum size=0.2cm}, scale=.5] 
     \tikzstyle{edge_style} = [draw=black]
     \node (v1) at (-3,0.5) [blackdot,label=below:$\tau_k''$]{};
     \node (v2) at (-2,3) [blackdot, label=above:{$\tau_{k+1}''=\sigma_i''$}]{};
     \node (v3) at (-1,2) [blackdot,]{};
      \node (v4) at (-0,1) [blackdot, label=right:{$\tau_{j}''$}]{};
     \node (v5) at (1,-1) [blackdot, label=below:{$\tau_{j+1}''$}]{};
     \node (v6) at (2,0) [blackdot]{};
     \draw[edge_style] (v1) to (v2);
     \draw[edge_style] (v2) to (v3);
     \draw[edge_style] (v3) to (v4);
     \draw[edge_style] (v4) to (v5);
     \draw[edge_style] (v5) to (v6);
\end{tikzpicture} \hspace{2cm}

\caption{A schematic drawing for the case when $\des(\delta)+2\leq x \leq |\delta|+1$  modulo $|\delta|+1$.}
\label{fig-desmaj2}
\end{figure}

This finishes the proof of the shuffle compatibility of $\des$ and $(\maj,\des)$. 
\end{proof}

In the above proof we proceeded to reduce permutations by moving descents to the left  as far as we could. However, to show the shuffle compatibility of $\maj$ we must reduce the permutations even further since permutations with the same value of $\maj$ may have different numbers of descents.

\begin{thm}
\label{thm-majsc}
The permutation statistic $\maj$ is shuffle compatible.
\end{thm}

\begin{proof}
As in the previous proof, our first step is to use Corollary~\ref{cor-screduction} to reduce to showing that $\pi \in L(m)$ and $\sigma, \sigma'\in L([n]+m)$ with $\maj(\sigma)=\maj(\sigma')$ implies $\maj(\pi\shuffle \sigma) = \maj(\pi \shuffle \sigma')$.
We will use the same canonical permutation for every element of $L([n]+m)$ by letting
$$\Pi = \{\sigma\in L([n]+m) \mid \Des(\sigma)=\emptyset\}. $$ 
This set contains the unique increasing permutation.
Our measure of how close a permutation is to being in $\Pi$ is $d: L([n]+m)\to \bbN$ given by $d(\sigma) =\maj(\sigma)$.   Observe that $\sigma\in \Pi$ if and only if $\maj(\sigma)=0$. Our strategy will be to find a bijection between shuffles $\tau\in\pi\shuffle\sigma$ and shuffles with the element of $\Pi$  which lowers the major index of each $\tau$ by the same amount, namely 
$\maj(\sigma)$.

To reduce a permutation to one in $\Pi$ we will move their descents to the left one position at a time until they are moved to position $0$ at which point they vanish. More precisely, if $\sigma \not\in \Pi$ then there exists at least one position $i\in \Des(\sigma)$ such that $i-1\not\in \Des(\sigma)$ and we  allow the case $i=1$. Fix any such $i$ and let $\sigma''$ be any permutation such that 
$$\Des(\sigma'') = 
\begin{cases}
\left(\Des(\sigma)\setminus \{i\}\right)\cup \{i-1\} & \text{ if } i\geq 2,\\
\Des(\sigma)\setminus \{1\} & \text{ if } i=1.
\end{cases} $$

The map $\Theta$ from the proof of Theorem~\ref{thm-desshufcompat} suffices when $i\geq 2$, so we need only give a bijection $\widetilde{\Theta}:\pi\shuffle \sigma \to \pi \shuffle \sigma ''$ for the case $i=1$ such that the image $\tau''=\widetilde{\Theta}(\tau)$ satisfies 
\begin{equation}
\label{tau''2}
\maj(\tau'') =\maj(\tau)-1.
\end{equation}
This means that if there is a descent at position $1$ then we need a bijection which reduces $\maj$ by one  by changing that descent to an ascent.
Set 
$$\widetilde{\sigma} = m+1,\sigma_1+1,\sigma_2+1,\ldots,\sigma_n+1\in L([n+1]+m)$$
and 
$$\widetilde{\sigma}'' = m+n+1,\sigma_1'',\sigma_2'',\ldots, \sigma_n''\in L([n+1]+m).$$

For a permutation $\sigma$, let
\begin{equation}
\label{eqn-ssigmadef}
S_\sigma = \{\tau\in \pi \shuffle \sigma \mid \tau_1=\sigma_1 \}
\end{equation}
be the subset of $\pi\shuffle \sigma$ whose elements all have $\sigma_1$ in the first position.
There is a natural bijection $\iota: \pi\shuffle \sigma \to S_{\widetilde{\sigma}}$. Namely, for $\tau\in \pi\shuffle \sigma$, let  
$\iota(\tau) = (m+1)\widetilde{\tau}$ where $\widetilde{\tau}$ is the unique permutation such that $\omega(\tau)=\omega(\widetilde{\tau})$, i.e., $\widetilde{\tau}$ is $\pi\shuffle\sigma$ with all elements of $\sigma$ increased by $1$. There is an analogous map $\iota'':\pi\shuffle \sigma''\rightarrow S_{\widetilde{\sigma}''}$.

Note that $\Theta(S_{\widetilde{\sigma}}) = S_{\widetilde{\sigma}''}$ since 
$$
\Des(\widetilde{\sigma}'')=(\Des(\widetilde{\sigma})\setminus\{2\})\cup \{1\}.
$$
It follows that we can define $\widetilde{\Theta}$ by insisting that the following diagram commutes
$$\begin{tikzpicture}
\node (A1) {$\pi\shuffle \sigma$};
\node (A2) [right of=A1, xshift=2cm] {$S_{\widetilde{\sigma}}$};
\node (A3) [below of=A1] {$\pi \shuffle \sigma''$};
\node (A4) [below of=A2] {$S_{\widetilde{\sigma}''}$};
\node (A5) [right of=A2, xshift=-1cm] {$\subseteq$};
\node (A6) [right of=A5, xshift=-0.8cm] {$\pi\shuffle\widetilde{\sigma}$};
\node (A7) [right of=A4, xshift=-1cm] {$\subseteq$};
\node (A8) [right of=A7, xshift=-0.8cm] {$\pi \shuffle \widetilde{\sigma}''$};
\draw[->](A1) to node{$\iota$}  (A2);
\draw[->] (A1) to [swap]node{$\widetilde{\Theta}$} (A3);
\draw[->] (A3) to node{$\iota''$} (A4);
\draw[->] (A2) to node{$\Theta$} (A4);
\end{tikzpicture}
$$
In other words, we define $\widetilde{\Theta} = \iota''^{-1}\circ\Theta \circ \iota$ which is clearly bijective.

To finish, it suffices to show that $\widetilde{\Theta}$ reduces $\maj$ by $1$.
First of all, observe that we have  
$$\maj(\iota(\tau)) = \begin{cases}
\maj(\tau)+\des(\tau) & \text{ if } \tau_1=\sigma_1\\
\maj(\tau)+\des(\tau)+1 & \text{ if } \tau_1=\pi_1\\
\end{cases}$$
 since each descent is shifted to the right by one position, and if $\tau_1=\pi_1$ there is an additional descent at position 1. By the previous theorem, $\maj(\Theta(\iota(\tau)))=\maj(\iota(\tau))-1$. Also 
$$\des(\Theta(\iota(\tau))) = \des(\iota(\tau))=\begin{cases}
\des(\tau) & \text{ if } \tau_1=\sigma_1\\
\des(\tau)+1 & \text{ if } \tau_1=\pi_1\\
\end{cases}$$
 since, by the previous theorem, $\Theta$ was $\des$ preserving.
Finally, for an element $\tau''\in \pi \shuffle \widetilde{\sigma}''$ we have
$$\maj(\iota''^{-1}(\tau'')) = \maj(\tau'')-\des(\tau'')$$
since each descent of $\tau''$ is moved to the left by one position in $\iota''^{-1}(\tau'')$.

Thus, regardless of the first element of $\tau$,
\begin{align*}
\maj(\iota''^{-1}(\Theta(\iota(\tau))) &= \maj(\Theta(\iota(\tau))-\des(\Theta(\iota(\tau))\\
&=(\maj(\iota(\tau))-1)-\des(\Theta(\iota(\tau))\\
&=\maj(\tau)-1.
\end{align*}
which finishes the proof.
\end{proof}

We can now  recover the identity for the distribution of $\maj$ over the shuffle set. 
We start with a  well-known result whose proof can be found in \cite{Bona2016} Section 2.2.2. B\'ona's treatment deals with rearrangements of a multiset containing ones and twos, which is easily seen to be equivalent to shuffling two words, the first consisting only of ones and the second only of twos. One defines $\maj$ in the same way for words with repeated numbers.

\begin{thm}
\label{thm-maj_ident}
Let $e =\overbrace{11\ldots 1}^m$ and $f=\overbrace{22\ldots 2}^n$ be words of length $m$ and $n$ respectively. Then
\begin{equation}
\label{eqn-base_maj_identity}
\sum \limits_{\tau \in e \shuffle f} q^{\maj(\tau)}= {m+n \brack m}_q.
\end{equation}
\end{thm}

The previous result will act as the base case for our inductive proof of the follow result cited in the introduction.

\begin{thm} Let $\pi$ and $\sigma$ be permutations with disjoint domains and lengths $m$ and $n$, respectively.  Then
$$\sum \limits_{\tau\in \pi\shuffle \sigma} q^{\maj\tau} = q^{\maj\pi+\maj\sigma}{m+n \brack m}_q.$$
\end{thm}

\begin{proof}
Since we have shown that $\maj$ is shuffle compatible, we may assume that $\pi\in L([m])$ and $\sigma\in L([n]+m)$.
We will induct on $\maj(\pi)+\maj(\sigma)$.  If $\maj(\pi)+\maj(\sigma)=0$ then $\pi=12\ldots m$ and 
$\sigma = m+1,m+2,\ldots,m+n$.  In this case, the result follows from Theorem \ref{thm-maj_ident}.
This is because  replacing $e$ with $\pi$ and $f$ with $\sigma$, respectively, in  $\tau\in e\shuffle \, f$ turns a repeated pair $11$ or $22$ into an ascent, while descents remain descents since all elements of $\sigma$ are larger than those of $\pi$.

Now assume $\maj(\pi)+\maj(\sigma)>0$. By  Lemma~\ref{lem-setcompat} we can assume, without loss of generality, that 
$\maj(\sigma)>0$.  So the map $\widetilde{\Theta}:\pi\shuffle \sigma \to \pi \shuffle \sigma ''$ of Theorem~\ref{thm-majsc} is a bijection, where 
$\maj(\sigma'')=\maj(\sigma)-1$ and $\maj(\tau'')=\maj(\tau)-1$ for $\tau''=\widetilde{\Theta}(\tau)$.  By induction, the desired equation holds for $\pi \shuffle \sigma ''$.  Multiplying the equality by $q$ and substituting, shows that it also holds for $\pi\shuffle\sigma$.

\end{proof}

\section{Peak Statistics}

We now move on to statistics related to peaks. The proof for the statistic $(\udr, \pk)$ is notable because it was previously only conjectured to be shuffle compatible and here we give a proof that is similar in nature to those for the other peak statistics.
\begin{thm}
\label{thm-pkshufcompat}
The statistic $\pk$ is shuffle compatible.
\end{thm}
\begin{proof}
By  Corollary~\ref{cor-screduction} part (2) it suffices to prove that if $\pi,\pi' \in L([m])$ and $\sigma \in L([n]+m)$ with $\pk(\pi)=\pk(\pi')$ then $\pk(\pi\shuffle \sigma) = \pk(\pi' \shuffle \sigma)$. For the permutations with $\pk\pi=p$ we will use the canonical set
$$\Pi = \{\pi \in L([m]) \mid \Pk(\pi) =\{2,4, \ldots, 2p\}\}$$
which contains exactly the permutations  with $p$ peaks which are as far to the left as possible.
So if $\pi,\pi'\in\Pi$ then $\Pk(\pi)=\Pk(\pi')$.  It follows that  $\Pk(\pi\shuffle \sigma)=\Pk(\pi'\shuffle \sigma)$ since $\Pk$ is shuffle compatible. Therefore $\pk(\pi\shuffle \sigma)=\pk(\pi'\shuffle \sigma)$ and the conclusion of (ii) of the general approach holds.

Our measure of how close a permutation is to being in $\Pi$ is $d: L([m])\to \bbN$ given by 
$$d(\pi) =\sum \limits_{k\in \Pk(\pi)}k.$$
 Note that among all permutations in $L([m])$ with $\pk\pi=p$, the ones in $\Pi$ have the minimum possible $d$, namely 
$p(p+1)$.  Our strategy will be to find a bijection between shuffles $\tau\in\pi\shuffle\sigma$ and shuffles with an element of $\Pi$ which preserves $\pk$ and lowers $d$ by the proper amount, namely $d(\pi)-p(p+1)$.  

To reduce a permutation to one in $\Pi$ we will move its peaks to the left one position at a time. More specifically, if $\pi \not\in \Pi$ then there is at least one position $j\geq 3$ such that $j\in \Pk(\pi)$, but $j-2\not\in \Pk(\pi)$.  Thus there exists $\pi''\in L([m])$ such that 
\begin{equation}
\label{Pkpi''}
\Pk(\pi'') = \bigg(\Pk(\pi)\setminus \{j\}\bigg) \cup \{j-1\}.
\end{equation}
Since $d(\pi'')<d(\pi)$, it suffices to give a $\pk$-preserving bijection between $\pi\shuffle \sigma$ and $\pi''\shuffle \sigma$. 

\begin{figure}

\centering
\begin{tikzpicture}
[blackdot/.style={circle, fill=black,thick, inner sep=2pt, minimum size=0.2cm}, scale=.5] 
     \tikzstyle{edge_style} = [draw=black];
     \node (v000) at (-5,0) [blackdot, label=below:{$\pi_{j-2}=\tau_s$}]{};
	 \node (v00) at (-4,3)  [blackdot]{};
	 \node (v0) at (-2, 3) [blackdot]{};    
     \node (v1) at (-1,0.5) [blackdot,label=below:$\pi_{j-1}$]{};
     \node (v4) at (1,2) [blackdot, label=above:$\pi_j$]{};
     \node (v7) at (3,-2) [blackdot, label=below:{$\pi_{j+1}=\tau_t$}]{};
     \draw[edge_style] (v000) to (v00);
     \draw[snake=zigzag] (v00) to node{$\sigma_a$}  (v0);
     \draw[edge_style] (v0) to (v1);
     \draw[edge_style] (v1) to (v4);
     \draw[edge_style] (v4) to (v7);
\end{tikzpicture} \hspace{2cm}
\begin{tikzpicture}
[blackdot/.style={circle, fill=black,thick, inner sep=2pt, minimum size=0.2cm}, scale=.5] 
     \tikzstyle{edge_style} = [draw=black];
     \node (v000) at (-3,0) [blackdot, label=below:$\pi_{j-2}''$]{};
     \node (v1) at (-1,2) [blackdot,label=below:$\pi_{j-1}''$]{};
     \node (v4) at (1,0.5) [blackdot, label=below:$\pi_j''$]{};
     \node (v5) at (3,3) [blackdot]{};
     \node (v6) at (5,3) [blackdot]{};
     \node (v7) at (7,-2) [blackdot, label=below:$\pi_{j+1}''$]{};
     \draw[edge_style] (v000) to (v1);
     \draw[edge_style] (v1) to (v4);
     \draw[edge_style] (v4) to (v5);
     \draw[snake=zigzag] (v5) to node{$\sigma_a$} (v6);
     \draw[edge_style] (v6) to (v7);
\end{tikzpicture}
\caption{An illustration of the case when $\sigma_a \not=\emptyset$ and $\sigma_b = \sigma_c = \emptyset$.  }
\label{fig-pkane}
\end{figure}
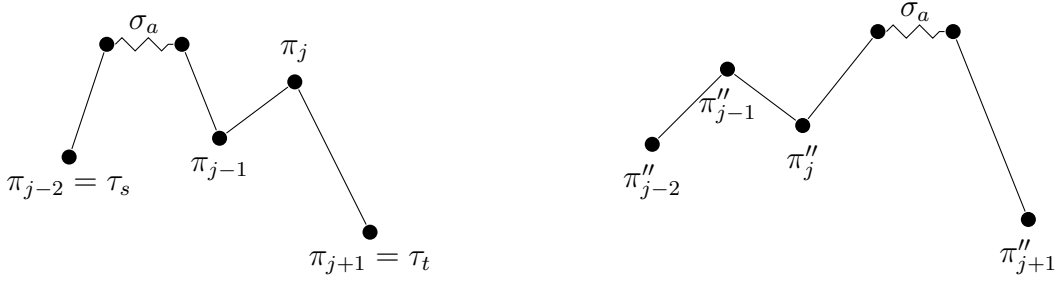

For each $\tau\in \pi\shuffle \sigma$, factor $\tau=\tau^a\tau^b\tau^c$ where $\tau^b$ is the factor
 of $\tau$ between $\pi_{j-2}$ and $\pi_{j+1}$, not including $\pi_{j-2}$ and $\pi_{j+1}$. Then $\tau^a$ is the remaining initial factor of $\tau$ and $\tau^b$ is the final factor.
Factor $\tau^b$ even further as 
$$\tau^b = \sigma^a \pi_{j-1}  \sigma^b \pi_j \sigma^c$$
so that $\sigma^a,\ \sigma^b,\ \sigma^c$ are the factors of $\sigma$ that are between the corresponding elements of $\pi$. Note that it is possible for any or all of $\sigma^a,\ \sigma^b,\ \sigma^c$ to be empty.

Define a map $\Theta:\pi \shuffle \sigma \to \pi''\shuffle \sigma$ by

\begin{equation}
\label{eqn-pkbij}
\Theta(\tau) = \begin{cases}
(\tau^a)''\pi_{j-1}''\pi_{j}''\sigma^a(\tau^c)'' & \text{ if } \sigma^a \not=\emptyset \text{ and } \sigma^b=\sigma^c=\emptyset, \\
\\
(\tau^a)''\sigma^c\pi_{j-1}''\pi_{j}''(\tau^c)'' & \text{ if } \sigma^a=\sigma^b=\emptyset \text{ and } \sigma^c\not=\emptyset, \\
\\
\Phi(\tau) & \text{ otherwise} \\
\end{cases}
\end{equation}
where $(\tau^a)''$ is the unique permutation such that $\omega(\tau^a)=\omega((\tau^a)'')$ and $(\tau^c)''$ is the unique permutation such that $\omega(\tau^c)=\omega((\tau^c)'')$.
It is clear from its definition that $\Theta$ is a  bijection.
So it only remains to show that $\Theta$ is $\pk$ preserving. 
Let $s,t$ be such that $\tau_s=\pi_{j-2}$ and $\tau_t=\pi_{j+1}$ and set $\ell = |\sigma^a|$.
Note that, as in the proof of Theorem~\ref{thm-shufcompatsetd}, for $i\in [m+n]\setm [s,t]$ we have $i\in \Pk(\tau)$ if and only if $i\in \Pk(\tau'')$.  So to  show that $\Theta$ is $\pk$ preserving we just need to concentrate on those peaks in $[s,t]$.

If $\sigma^a \not=\emptyset$ and  $\sigma^b=\sigma^c=\emptyset$ then it is straightforwards to check, using the cases from the proof of Theorem~\ref{thm-shufcompatsetd} for $\Pk$ that the only peaks of $\tau$ in the set $[s,t]$ occur as one of the following.
\begin{enumerate}
\item[(a)] Every peak of $\sigma^a$ is a peak of $\tau$.
\item[(b)] $s+1\in \Pk(\tau)$ if and only if $1\in \Des(\sigma^{a})$ or $\ell=1$.
\item[(c)] For $\ell\ge2$: $s+\ell \in \Pk(\tau)$ if and only if $\ell-1\in \Asc(\sigma^a)$.
\item[(d)] $t-1$ is always a peak of $\tau$.
\end{enumerate}

We now compare this to the similar list for $\tau''$. 
\begin{enumerate}
\item[(a)] Every peak of $\sigma^a$ is a peak $\tau''$.
\item[(b)] $s+3\in \Pk(\tau'')$ if and only if $1\in \Des(\sigma^a)$ or $\ell=1$.
\item[(c)]  For $\ell \geq 2$: $t-1\in \Pk(\tau'')$ if and only if $\ell-1 \in \Asc(\sigma^a)$.
\item[(d)]$s+1$ is always a peak of $\tau''$.
\end{enumerate}
Clearly these lists  contain the same number of peaks. An illustration of this case is given in Figure~\ref{fig-pkane}. In the figure, jagged lines represent a part of $\sigma$.

\begin{figure}

\centering
\begin{tikzpicture}
[blackdot/.style={circle, fill=black,thick, inner sep=2pt, minimum size=0.2cm}, scale=.47] 
     \tikzstyle{edge_style} = [draw=black];
     \node (v000) at (-5,0) [blackdot, label=below:{$\pi_{j-2}=\tau_s$}]{};
	 \node (v00) at (-4,3)  [blackdot]{};
	 \node (v0) at (-2, 3) [blackdot]{};    
     \node (v1) at (-1,0.5) [blackdot,label=below:$\pi_{j-1}$]{};
     \node (v2) at (0,3) [blackdot]{};
     \node (v3) at (2,3) [blackdot]{};
     \node (v4) at (3,2) [blackdot, label=below:$\pi_j$]{};
     \node (v5) at (4,3) [blackdot]{};
     \node (v6) at (6,3) [blackdot]{};
     \node (v7) at (7,-2) [blackdot, label=below:{$\pi_{j+1}=\tau_t$}]{};
     \draw[edge_style] (v000) to (v00);
     \draw[snake=zigzag] (v00) to node{$\sigma_a$}  (v0);
     \draw[edge_style] (v0) to (v1);
     \draw[edge_style] (v1) to (v2);
     \draw[snake=zigzag] (v2) to node{$\sigma_b$} (v3);
     \draw[edge_style] (v3) to (v4);
     \draw[edge_style] (v4) to (v5);
     \draw[snake=zigzag] (v5) to node{$\sigma_c$} (v6);
     \draw[edge_style] (v6) to (v7);
\end{tikzpicture} \hspace{2cm}
\begin{tikzpicture}
[blackdot/.style={circle, fill=black,thick, inner sep=2pt, minimum size=0.2cm}, scale=.47] 
     \tikzstyle{edge_style} = [draw=black];
     \node (v000) at (-5,0) [blackdot, label=below:{$\pi_{j-2}''$}]{};
	 \node (v00) at (-4,3)  [blackdot]{};
	 \node (v0) at (-2, 3) [blackdot]{};    
     \node (v1) at (-1,2) [blackdot,label=below:$\pi_{j-1}''$]{};
     \node (v2) at (0,3) [blackdot]{};
     \node (v3) at (2,3) [blackdot]{};
     \node (v4) at (3,0.5) [blackdot, label=below:$\pi_j''$]{};
     \node (v5) at (4,3) [blackdot]{};
     \node (v6) at (6,3) [blackdot]{};
     \node (v7) at (7,-2) [blackdot, label=below:$\pi_{j+1}''$]{};
     \draw[edge_style] (v000) to (v00);
     \draw[snake=zigzag] (v00) to node{$\sigma_a$}  (v0);
     \draw[edge_style] (v0) to (v1);
     \draw[edge_style] (v1) to (v2);
     \draw[snake=zigzag] (v2) to node{$\sigma_b$} (v3);
     \draw[edge_style] (v3) to (v4);
     \draw[edge_style] (v4) to (v5);
     \draw[snake=zigzag] (v5) to node{$\sigma_c$} (v6);
     \draw[edge_style] (v6) to (v7);
\end{tikzpicture}
\caption{An illustration of the case when $\sigma_a, \sigma_b$, and $\sigma_c$ are all nonempty.}
\label{fig-pkne}
\end{figure}
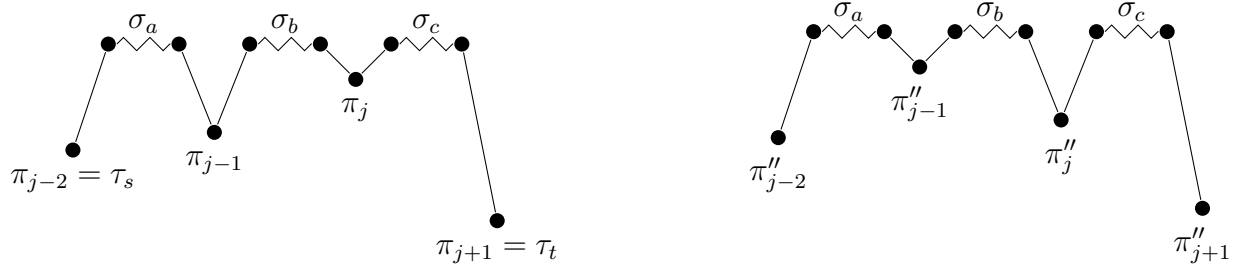

Next note that if $\sigma^a=\sigma^b=\emptyset$ and $\sigma^c\not=\emptyset$, then these two lists are swapped. So $\Theta$ is $\pk$ preserving in this case as well.

Now if $\tau^b$ is not in one of the previous two cases and $\sigma^a,\sigma^b$, and $\sigma^c$ are not all simultaneously empty then one  can check lists similar to those above to see that the peaks of both $\tau$ and $\tau''$ in the range $[s,t]$ are exactly the peaks of $\sigma^a, \sigma^b$, and $\sigma^c$ together with possibly their endpoints. An illustration of this case is given in Figure~\ref{fig-pkne}. Set $\ell_a=|\sigma^a|$, $\ell_b=|\sigma^b|$, and $\ell_c=|\sigma^c|$ and suppose all three cardinalities are nonzero.

\begin{enumerate}
\item[(a)] Every peak of $\sigma^a, \sigma^b$ and $\sigma^c$ is a peak of $\tau$.
\item[(b)] $s+1\in \Pk(\tau)$ if and only if $1\in \Des(\sigma^a)$ or $\ell_a=1$.
\item[(c)] For $\ell_a\geq 2$: $s+\ell_a \in \Pk(\tau)$ if and only if $\ell_a-1 \in \Asc(\sigma^a)$.
\item[(d)] $s+\ell_a+2\in \Pk(\tau)$ if and only if $1\in \Des(\sigma^b)$ or $\ell_b=1$.
\item[(e)] For $\ell_b\geq 2$: $s+\ell_a+\ell_b+1 \in \Pk(\tau)$ if and only if $\ell_b-1 \in \Asc(\sigma^b)$.
\item[(d)] $s+\ell_a+\ell_b+3\in \Pk(\tau)$ if and only if $1\in \Des(\sigma^c)$ or $\ell_c=1$.
\item[(e)] For $\ell_c\geq 2$: $s+\ell_a+\ell_b+\ell_c+2 \in \Pk(\tau)$ if and only if $\ell_c-1 \in \Asc(\sigma^c)$.
\end{enumerate}
The list for $\tau''$ is identical.

Finally, if $\sigma^a=\sigma^b=\sigma^c=\emptyset$, then $\Pk(\tau) \cap[s,t]= \{t-1\}$ and $\Pk(\tau'')\cap[s,t]= \{t-2\}$. Hence the number of peaks is again preserved. 
\end{proof}

\begin{thm}
\label{thm-lpksc}
The statistics $\lpk$, $\rpk$, $\epk$, $\udr$, and $(\udr,\pk)$  are shuffle compatible. 
\end{thm}
The proofs for these statistics are based on the same idea as that of Theorem~\ref{thm-pkshufcompat}, but additional variants of the bijection used there are needed.
We again  use Corollary~\ref{cor-screduction} part (2).  So it suffices to show that if $\pi,\pi' \in L([m])$ and $\sigma \in L([n]+m)$ with $\St(\pi)=\St(\pi')$ then $\St(\pi\shuffle \sigma) = \St(\pi' \shuffle \sigma')$ for each $\St\in \{ \lpk, \rpk, \epk, \udr, (\udr,\pk)\}$. The main tools for this proof are the bijection $\Theta$ of Theorem~\ref{thm-pkshufcompat} and other similar bijections that preserve these statistics and allow us to replace $\pi$ with a permutation that is closer to being in our chosen set of canonical permutations, as outlined in the general approach.  
\vspace{.2cm}

\noindent\textbf{Proof for $\lpk$: }\\
To obtain the shuffle compatibility of $\lpk$ we reduce permutations  with $\lpk \pi=p$ to the canonical set
$$\Pi = \{\pi \in L([m]) \mid \Lpk(\pi) =\{1,3,5, \ldots, 2p-1\}\}. $$
The proof that two permutations in this set have the same $\lpk$ is similar to the analogous statement for $\pk$ and so is omitted.
We use a measure $d$ similar to that for $\pk$, but summing over left peaks instead.
So the minimal value for a permutation with $\lpk(\pi)=p$ is 
$$d(\pi) =\sum \limits_{k\in \Lpk(\pi)}k =  p^2.$$

If $\pi\not \in \Pi$ then there exists a position $j\geq 2$ such that $j\in \Lpk(\pi)$, but $j-2\not\in \Lpk(\pi)$. Let $\pi''$ be any permutation such that 
\begin{equation}
\label{eqn-lpkpi''}
\Lpk(\pi'') = \left(\Lpk(\pi)\setminus \{j\}\right)\cup \{j-1\}.
\end{equation}

Then it suffices to give a bijection $\widetilde{\Theta}:\pi\shuffle \sigma \to \pi''\shuffle \sigma$ that reduces $d$ and is $\lpk$ preserving.
If $j\geq 3$, then the bijection of the above proof for $\pk$ suffices. Thus, assume $j=2$.
To construct $\widetilde{\Theta}$ we proceed in a manner similar to the proof of Theorem~\ref{thm-majsc}.

Set $\widetilde{\pi}=0\pi$ and $\widetilde{\pi}''=0\pi''$ and use the notation 
$$S_\pi = \{\tau\in \pi\shuffle \sigma\; |\; \tau_1=\pi_1\}.$$
 Then we have the following commutative diagram.

$$\begin{tikzpicture}
\node (A1) {$\pi\shuffle \sigma$};
\node (A2) [right of=A1, xshift=2cm] {$S_{\widetilde{\pi}}$};
\node (A3) [below of=A1] {$\pi \shuffle \sigma''$};
\node (A4) [below of=A2] {$S_{\widetilde{\pi}''}$};
\node (A5) [right of=A2, xshift=-1cm] {$\subseteq$};
\node (A6) [right of=A5, xshift=-0.8cm] {$\widetilde{\pi}\shuffle \sigma$};
\node (A7) [right of=A4, xshift=-1cm] {$\subseteq$};
\node (A8) [right of=A7, xshift=-0.8cm] {$\widetilde{\pi} \shuffle \sigma''$};
\draw[->](A1) to node{$\iota$}  (A2);
\draw[->] (A1) to [swap]node{$\widetilde{\Theta}$} (A3);
\draw[->] (A3) to node{$\iota''$} (A4);
\draw[->] (A2) to node{$\Theta$} (A4);
\end{tikzpicture}
$$

Here $\iota$ is the bijection which identifies $\tau\in \pi\shuffle \sigma$ with $0\tau \in S_{\widetilde{\pi}}$. 
The map $\iota''$ is defined similarly. The map $\Theta$ is the $\pk$-preserving bijection used in the proof of $\pk$ to move the peak at position $3$ to position $2$.

Then 
\begin{equation}
\label{eqn-lpkbij}
\widetilde{\Theta}= \iota''^{-1} \circ \Theta \circ \iota
\end{equation}
 is the required bijection.  It is clear that this map reduces $d$ by $1$ since $\Theta$ has this property. The injection $\Theta\circ \iota$ is $\lpk$ preserving because $\Theta$ is $\pk$ preserving and position 2 in $\iota(\tau)$ is a peak if and only if position 1 is a left peak of $\tau$. And similarly, position 1 in  $\iota''^{-1}\Theta(\iota(\tau)))$ is a left peak if and only if it position 2 is a peak in  $\Theta(\iota(\tau))$ which proves the claim and completes the demonstration for $\lpk$.
 \hfill $\square$
\vspace{.2cm}

\noindent\textbf{Proof for $\rpk$: }\\
To obtain the shuffle compatibility of $\rpk$ we use an approach similar to that of $\lpk$ by changing our set of canonical permutations 
for $\pi$ with $\rpk(\pi)=p$ to
$$\Pi = \{\pi \in L([m]) \mid \Rpk(\pi) =\{m,m-2,m-4, \ldots, m-2p\}\}. $$
We also change our measure $d$ of how close a permutation is to being in $\Pi$ to
$$d(\pi) = \sum \limits_{k\in \Rpk(\pi)} (m-k).$$
 We have that the minimal value for a permutation with $\rpk(\pi)=p$ is $d(\pi) = p(p+1)$
and if two permutations $\pi_1, \pi_2\in \Pi$ satisfy $\Rpk(\pi_1)=\Rpk(\pi_2)$, then by Theorem~\ref{thm-shufcompatsetd} they satisfy the conclusion of (ii) of the general approach.

If $\pi\not \in \Pi$ then there exists a position $j\leq m-1$ such that $j\in \Rpk(\pi)$, but $j+2\not\in \Rpk(\pi)$. Let $\pi''$ be any permutation such that 
$$\Rpk(\pi'') = \left(\Rpk(\pi)\setminus \{j\}\right)\cup \{j+1\}.$$

The remainder  of the proof follows the same lines as for $\lpk$ except we move peaks to the right instead of left using the inverse bijections. \hfill $\square$

\vspace{.2cm}

\noindent\textbf{Proof for $\epk$: }\\
The shuffle compatibility of $\epk$ follows  from the combination of bijections in the proofs for $\pk$, $\lpk$, and $\rpk$. We use the the canonical set of permutations
$$\Pi = \{\pi \in L([m]) \mid \Epk(\pi) =\{1,3 ,5, \ldots, 2p-1\} \text{ for some } p\geq 0\}$$
and measure 
$$d(\pi) = \sum \limits_{k\in \Epk(\pi)}k.$$
We move all peaks or right peaks as far to the left as possible using the bijections from $\pk$, $\lpk$ and inverse of the bijection from $\rpk$ can be used to move a final ascent to the left. The remainder of the proof is analogous to those of $\lpk$ and $\rpk$. \hfill $\square$

\vspace{.2cm}

\noindent\textbf{Proof for $\udr$: }\\
As a first step, we observe that for a permutation $\pi\in L([m])$, $m\ge1$, we have  

\begin{equation}
\label{eqn-udr}
\udr(\pi) = 
2\pk(0\pi)+\chi^+(0\pi) 
\end{equation}
 since every peak of $0\pi$  involves two distinct runs, and $\chi^+(0\pi)$ accounts for the possibilities of either a final increasing run or that $|\pi|=1$. There is nothing to prove if $\pi=\emptyset$ and the proof for $|\pi|=1$ is trivial, so we will assume for the remainder of this proof that $m\geq 2$. In this case equation~(\ref{eqn-udr}) simplifies slightly to
\begin{equation}
\label{eqn-udr2}
\udr(\pi) = 
2\lpk(\pi)+\chi^+(\pi)
\end{equation} 
since the left peaks of $\pi$ are the peaks of $0\pi$.
Considering this equation modulo two we see that the value of $\udr(\pi)$ determines both $\lpk(\pi)$ and $\chi^+(\pi)$, as well as conversely.

Just as for $\lpk$, our canonical set for permutations with a $\udr\pi=2p+\chi^+(\pi)$ is
$$\Pi = \{\pi\in L([m]) \mid \Lpk(\pi) = \{1, 3, 5, \ldots 2p-1\}\}. $$
Note this is the same canonical set as was used in the proof of $\lpk$.
Take two permutations $\pi, \pi'\in \Pi$ with the same $\udr$.  So, as discussed in the previous paragraph, $\chi^+(\pi)=\chi^+(\pi')$.   It follows that
$\Lpk(\pi)=\Lpk(\pi')$.   So for any $\sigma\in L([n]+m)$ we have
$$\udr(\pi\shuffle \sigma) = \udr(\pi'\shuffle \sigma)$$ 
since the bijection $\Phi$  preserves both $\Lpk$ and $\chi^+$. Therefore part $(ii)$ of the general approach is satisfied.

Our measure of how close a permutation is to being in the canonical set will be the same as it was for $\lpk$,
$$d(\pi) =\sum \limits_{k\in \Lpk(\sigma)}k.$$
A permutation with $\lpk(\pi)=p$ is in this canonical set if and only if $d(\pi) = p^2$, which is the minimal value for a permutation with $\udr(\pi)=2p+\chi^+(\pi)$.  Since $\udr(\pi)$ is determined by $\lpk(\pi)$ and $\chi^+(\pi)$, to complete the proof
 it will suffice to show that the bijection $\widetilde{\Theta}$ from (\ref{eqn-lpkbij})  used in the demonstration for $\lpk$ also preserves the statistic $\chi^+$. Note first that by the definition of $\widetilde{\Theta}$ that it suffices to check that $\Theta$ itself is $\chi^+$ preserving. This is because $\iota$ and $\iota''$ essentially prepend a $0$ to a permutation and then remove it. This does not affect the order relation between the final two elements of a shuffle since we have assumed that $|\pi|\geq 2$.
	
Since  definition~(\ref{eqn-pkbij}) for $\Theta$ uses the bijection $\Phi$, we first show that $\Phi$ is $\chi^+$ preserving  when applied to $\pi, \pi'' \in L([m])$  that satisfy $\udr(\pi)=\udr(\pi'')$.  
As previously noted, the assumption about $\udr$ implies $\chi^+(\pi)  = \chi^+(\pi'')$. Thus, the final two positions of $\pi$ and $\pi''$ must satisfy the same order relation. It follows that for a shuffle $\tau\in \pi\shuffle\sigma$, that replacing $\pi$ with $\pi''$  to obtain $\tau'=\Phi(\tau)$ does not change the order relation of the final two positions of $\tau$. This means that $\chi^+$ is also preserved under $\Phi$.

Now  assume that  $\pi\not\in \Pi$. Choose $\pi''$ as in equation~(\ref{eqn-lpkpi''}) with the additional restriction that
$\udr(\pi)=\udr(\pi'')$.
This implies $\chi^+(\pi'')=\chi^+(\pi)$. Let $\tau\in \pi\shuffle \sigma$. If $j\leq m-2$ or  $\tau_{m+n}\in \sigma$, it is clear from the definition of $\Theta$ given in (\ref{eqn-pkbij}) that $\chi^+(\tau)=\chi^+(\widetilde{\Theta}(\tau))$ and $\chi^+$ is preserved.

It therefore remains to check that $\Theta$ is $\chi^+$ preserving in the case that $j=m-1$ and $\tau_{m+n}=\pi_m$. 
Since we have already checked that $\Phi$ preserves $\chi^+$, we only need to deal with the first two cases in the definition of $\Theta$.
 \begin{itemize}

\item If $\sigma^a\not=\emptyset$, $\sigma^b=\sigma^c=\emptyset$, then we have
 $$\sigma^a\pi_{m-2} \pi_{m-1}\pi_{m} \stackrel{\Theta}{\mapsto} \pi_{m-2}'' \pi_{m-1}''\sigma^a \pi_{m}''$$
 so that both shuffles have a descent at position $m+n-1$.

\item If $\sigma^a=\sigma^b=\emptyset$,  $\sigma^c\not=\emptyset$, then we have 
 $$\pi_{m-2} \pi_{m-1}\sigma^c\pi_{m} \stackrel{\Theta}{\mapsto} \sigma^c\pi_{m-2}'' \pi_{m-1}''\pi_{m}''.$$
Since $\pi_{m-1}$ is a peak we have $\chi^+(\pi'')=\chi^+(\pi)=0$.  So, again, both shuffles  have a descent at position $m+n-1$

 \end{itemize}

From this we can conclude that the bijections $\Theta$, and hence $\widetilde{\Theta}$, used in the proofs for the statistics $\pk$ and $\lpk$ respectively are also $\udr$ preserving. \hfill $\square$

\vspace{.2cm}
\noindent\textbf{Proof for $(\udr, \pk)$: }\\
For any  permutation $\pi$ with $|\pi|\geq 2$, we can write
$$
\lpk(\pi)=\pk(\pi)+\chi^-(\pi)
$$
So~(\ref{eqn-udr2}) becomes
\begin{equation}
\udr(\pi) = 2\pk(\pi)+2\chi^-(\pi)+ \chi^+(\pi).
\end{equation}
By a parity argument like the one used for~(\ref{eqn-udr2}) we see that the value of  $(\udr(\pi), \pk(\pi))$ uniquely determines 
both $\chi^-$, and $\chi^+$.

Let
\begin{align*}
\Pi_0 &= \{\pi\in L([m]) \mid \Lpk(\pi) = \{ 2, 4, \ldots 2p\} \text{ for some } p\geq 0 \}\\
\Pi_1 &= \{\pi\in L([m]) \mid \Lpk(\pi) = \{1, 3, 5, \ldots 2p-1\} \text{ for some } p\geq 1 \}
\end{align*}
We then use the canonical set the disjoint union
$$\Pi = \Pi_0\sqcup \Pi_1.$$
Note that if $\pi, \pi'\in\Pi$ satisfy $(\udr(\pi), \pk(\pi))=(\udr(\pi'), \pk(\pi'))$ then, by the observation in the first paragraph of the proof, they are either both in $\Pi_0$ or both in $\Pi_1$.
It follows that $\Lpk(\pi)=\Lpk(\pi')$.  Now apply the bijection $\Phi:\pi\shuffle\sigma\rightarrow\pi'\shuffle\sigma$ where we have shown earlier that 
$\Phi$ preserves $\Lpk$ and $\chi^+$.  Also, the assumption on $\pi$ and $\pi'$ implies $\chi^-(\pi)=\chi^-(\pi')$.  It is easy to prove that in this case $\Phi$ preserves $\chi^-$.  It follows that $(\udr,\pk)$ is preserved by $\Phi$ and part (ii) of our method is satisfied.

Now set 
$$d(\pi) = \sum \limits_{k\in \Lpk(\pi)}k.$$
If $\pi\not\in\Pi$ then we use the map $\Theta$ from the proof for $\pk$ to map $\pi\shuffle\sigma$ to $\pi''\shuffle\sigma$ where $\pi''$ is given by~(\ref{Pkpi''}), as long as $j$ can be chosen with $j>3$, or $j=3$ and $\pi_1<\pi_2$.  
If the only possible $j$ value is $j=3$ and $\pi_1>\pi_2$ then we do not apply $\Theta$,
and we do not need to do so since $\pi\in\Pi_1$.
It follows that we can always choose $\pi''$ so that  $\chi^-(\pi)=\chi^-(\pi'')$.  One can now show that in this case $\Theta$ preserves 
$\chi^-$ similarly to the proof that the map preserves $\chi^+$.  Since $\Theta$ also preserves $\pk$, it preserves the pair $(\udr,\pk)$ and we are done.
\hfill $\square$

\section{Future Work}

There are still many open questions to be answered in this relatively new line of inquiry.  The first natural question is whether a proof similar to those above can be given for the statistic $(\udr, \pk, \des)$, which was conjectured to be shuffle compatible in \cite{GesZhu17}. 
\begin{quest}
Can a bijective proof for the statistic $(\udr, \pk, \des)$ be given that follows the general approach given in this article?
\end{quest}
Such a proof will require a more careful analysis of the common aspects of the bijection used for $\des$ and for $\pk$. The fundamental obstacle in our approach to this triple statistic is that our bijection for $\pk$ requires moves that may not preserve the number of descents in the permutation.

Other notions of shuffle compatiblitly were introduced by Grinberg in \cite{Grin2018}, so one could ask  whether the same general approach can be used to give bijective proofs for his shuffle compatibility analogues. One example is as follows.
\begin{defn}
A permutation statistic $\St$ is called \emph{left shuffle compatible} if  for any two disjoint nonempty permutations $\pi$ and $\sigma$ with the property that $\pi_1>\sigma_1$, the distribution
$$\{\{\St(\tau) \mid \tau \in \pi \shuffle \sigma,\ \tau_1=\pi_1 \; \}\}$$
depends only on $|\pi|$, $|\sigma|$, $\St(\pi)$, and $\St(\sigma)$. \hfill $\square$
\end{defn}
An analogous definition can be given for right shuffle compatibility. 
Note that an analogous set of shuffles appeared naturally in equation~\ree{eqn-ssigmadef}.
Our theory here would need to be modified as even Lemma~\ref{lem-setcompat} no longer holds. The bijections used there can take a shuffle starting with $\tau_1=\pi_1$ and swap it with one with $\tau_1=\sigma_1$ which is no longer in the set of left shuffles.

Another possible extension of this work is to permit permutations with repeated elements. For example consider the statistic $\maj$. Then 
$\pi = 4212$, $\pi' = 2221$, $\sigma=76$ and $\sigma'= 98$
satisfy $\maj(\pi)=\maj(\pi')$ and $\maj(\sigma)=\maj(\sigma')$.  Note that in this case
$$
\maj(\pi\shuffle \sigma) = \maj(\pi'\shuffle \sigma') = \{\{4, 5, 6^2, 7^2, 8^3, 9^2, 10^2, 11, 12\}\}.
$$
There is a canonical  extension of the standardization map to permutations with repetitions by replacing each maximal constant subword of $\pi$ with the elements $i,i+1,\dots,j$ for some $i,j$ from left to right.  For example,
$\std 24212 = 25314$.  This could be a tool in proving shuffle compatibility for permutations with repetitions.

The existence of a shuffle compatible permutation statistic that is not a descent statistic as constructed by O\u{g}uz in \cite{Oguz2018} raises the question as to how one would approach giving bijective proofs to such statistics. The proof for the statistic in \cite{Oguz2018} is by an exhaustive computation for all permutations of length 4 or less. Note that for non-descent statistics our approach no longer works. Indeed, Lemma \ref{lem-setcompat} only holds for descent statistics and so we lose the power of Corollary \ref{cor-screduction}.

\nocite{*}
\bibliographystyle{abbrvnat}



\end{document}